\documentclass[11pt,a4paper,final]{amsart}

\usepackage{amsfonts,amsmath,amssymb,indentfirst,mathrsfs,amscd}
\usepackage[mathscr]{eucal}
\usepackage[active]{srcltx}
\usepackage{graphicx}
\usepackage{nicefrac}
\usepackage{color}
\usepackage[utf8]{inputenc}
\usepackage[spanish, english]{babel}
\usepackage{booktabs}
\usepackage{multirow}
\usepackage{verbatim}
\usepackage{calrsfs}
\usepackage[all]{xy}
\usepackage{graphicx}
\usepackage{stmaryrd}
\usepackage{enumitem}
\usepackage{empheq}
\usepackage{ dsfont }

\usepackage{hyperref}

\DeclareMathAlphabet{\omscal}{OMS}{zplm}{m}{n}
\usepackage[all]{xy}
\usepackage{stmaryrd}
\usepackage{dsfont}

\DeclareMathOperator{\Hom}{Hom\,}           
\DeclareMathOperator{\Id}{Id\,}             
\DeclareMathOperator{\tr}{tr}             
\DeclareMathOperator{\Aut}{Aut\,}           

\DeclareMathOperator{\Gr}{Gr}
\DeclareMathOperator{\Spec}{Spec\,}
\newcommand{\bk}{{\mathbb{K}}}

\newcommand{\cA}{\omscal{A}}
\newcommand{\cC}{\omscal{C}}
\newcommand{\cD}{\omscal{D}} 
\newcommand{\cF}{\omscal{F}} 
\newcommand{\cO}{\omscal{O}} 
\newcommand{\cM}{\omscal{M}} 
\newcommand{\cU}{\omscal{U}} 
\newcommand{\cX}{\omscal{X}}
\newcommand{\cZ}{\omscal{Z}}

\newcommand{\cY}{\omscal{Y}}

\newcommand{\x}{\times}
\newcommand{\Cos}[1]{\widehat{\cM}_{#1}}

\newcommand{\CC}{\mathbb{C}} 
\newcommand{\QQ}{\mathbb{Q}} 
\newcommand{\PP}{\mathbb{P}} 
\newcommand{\FF}{\mathbb{F}} 
\newcommand{\ZZ}{\mathbb{Z}} 

\newcommand\CBord[3]{\mathbf{Bd}_{{#1 #3}}^{#2}}
\newcommand\CBordp[1]{\CBord{#1}{}{}}
\newcommand\CBordpp[1]{\mathbf{Bdp}_{{#1}}}

\newcommand\Mod[1]{{#1}\textrm{-}\mathbf{Mod}}

\newcommand\MHS[1]{\mathbf{MHS}}
\newcommand\Var[1]{\mathbf{Var}_{#1}}

\newcommand\K[1]{\mathrm{K}#1}

\newcommand\Varrel[1]{\mathbf{Var}/{#1}}

\newcommand\Rep[1]{\mathfrak{X}_{#1}}
\newcommand\Unit[1]{\mathds{1}_{#1}}
\newcommand\I{\mathrm{I}}
\newcommand{\G}{\Gamma}

\newcommand\set[1]{\left\{#1\right\}}

\newcommand\Aff[2]{\mathrm{AGL}_{#1}(#2)}
\newcommand\ASO[2]{\mathrm{ASO}_{#1}(#2)}
\newcommand\GL{\mathrm{GL}}
\newcommand\PGL{\mathrm{PGL}}
\newcommand\SL{\mathrm{SL}}
\newcommand\UU{\mathrm{U}}

\newtheorem{theorem}{Theorem}[section]

\newtheorem{corollary}[theorem]{Corollary}
\theoremstyle{definition}
\newtheorem{definition}[theorem]{Definition}
\theoremstyle{remark}
\newtheorem{remark}[theorem]{Remark}
\newtheorem{example}[theorem]{Example}

\title{Representation variety for the rank one affine group}

\author[A. Gonz\'alez-Prieto, M. Logares and V. Mu\~noz]{\'Angel Gonz\'alez-Prieto, Marina Logares and Vicente Mu\~noz}

\address{ETSI Sistemas Inform\'aticos, Universidad Polit\'ecnica de Madrid, Calle Alan Turing s/n, 28031 Madrid, Spain}
\email{angel.gonzalez.prieto@upm.es}
\address{Facultad de Ciencias Matem\'aticas, Universidad Complutense de Madrid, Plaza Ciencias 3, 28040 Madrid Spain.}\email{mlogares@ucm.es}
\address{Departamento de \'Algebra, Geometr\'ia y Topolog\'ia, Facultad de Ciencias, Universidad de M\'alaga, Campus de Teatinos s/n, 29071 Málaga, Spain}\email{vicente.munoz@uma.es}

\begin{document}

\let\thefootnote\relax
%
%

\maketitle

\begin{abstract}
The aim of this paper is to study the virtual classes of representation varieties of surface groups onto the rank one affine group. We perform this calculation by three different approaches: the geometric method, based on stratifying the representation variety into simpler pieces; the arithmetic method, focused on counting their number of points over finite fields; and the quantum method, which performs the computation by means of a Topological Quantum Field Theory. We also discuss the corresponding moduli spaces of representations and character varieties, which turn out to be non-equivalent due to the non-reductiveness of the underlying group.
\end{abstract}


\section{Introduction} \label{sec:1}

Let $\G$ be a finitely presented group and $G$ a complex algebraic group. A representation of $\G$ into $G$ is a group homomorphism 
$\rho:\G\longrightarrow G$. We shall denote the set of representations by 
 $$
  \Rep{G}(\G)= \Hom(\G,G), 
 $$
which is a complex algebraic variety. Let $X$ be a connected CW-complex with
$\pi_1(X)=\G$. Then $\Rep{G}(\G)$ parametrizes \emph{local systems} over $X$, that is, $G$-principal bundles $P \to X$ which
admit trivializations $P|_{U_\alpha} \simeq U_\alpha\x G$, for a covering $X=\bigcup U_\alpha$, such that the changes of charts
are (locally) constant functions $g_{\alpha\beta}:U_\alpha \cap U_\beta \to G$. A local system can also be understood as a covering space
with fiber $G$ (with the discrete topology). From another perspective, we can take a principal $G$-bundle $P \to X$ and fix a base point $x_0\in X$. Then a local system is equivalent to a \emph{flat connection} on $P$. Certainly, 
a flat connection $\nabla$ on $P$ determines the \emph{monodromy} representation $\rho_\nabla: \pi_1(X,x_0) \to \Aut(P_{x_0}) \cong G$,
given by associating to a path $[\gamma]\in \pi_1(X,x_0)$ the holonomy of $\nabla$ along $\gamma$. 
Finally, if $G$ admits a faithful representation $\kappa: G\hookrightarrow \GL_r(\CC)$, this can also be done with the vector bundle $E=P \x_{\kappa} \CC^r\to X$ with $G$ structure.

If we forget the trivialization at the base point, then we have the \emph{coset space} 
 \begin{equation}\label{eqn:coset}
 \Cos{G}(\G)=\Rep{G}(\G)/G ,
 \end{equation}
which is a topological space with the quotient topology. The action of $G$ changes the isomorphism
$ \Aut(P_{x_0}) \cong G$, which corresponds to the action of $G$ on $P$ as principal bundle. This induces the
adjoint action on the monodromy representation.
The space (\ref{eqn:coset}) parametrizes isomorphism classes of local systems. In this case we can
forget the base point, due to the isomorphisms $\pi_1(X,x_0)\cong \pi_1(X,x_1)$, for two points $x_0,x_1\in X$.
In general, the coset space is badly-behaved. It is not an algebraic variety, and it may be non-Hausdorff. 
From the algebro-geometric point of view, it is more natural to focus on the \emph{moduli space} of representations
$\cM_G(\G)$. This is defined as an algebraic variety with a ``quotient map'' $q:\Rep{G}(\G) \to \cM_G(\G)$ such that: (a)
$q$ is constant along orbits, that is $q$ is $G$-invariant; (b) it is an initial object for this property, that is any other
map $f:\Rep{G}(\G) \to Y$ which is $G$-invariant factors through $\cM_G(\G)$. It turns out that the moduli space
is defined by the \emph{GIT quotient}
 $$
 \cM_G(\G)=\Spec \cO(\Rep{G}(\G))^G\, ,
 $$
that is, its ring of functions is given by the $G$-invariant functions on the representation variety. 

In the case where $G$ is a complex reductive group (e.g.\ $G=\SL_r(\CC)$ or $\GL_r(\CC)$), the GIT quotient has
nice properties. Take a faithful representation $\kappa: G\hookrightarrow \GL_r(\CC)$.  The natural map
 \begin{equation}\label{eqn:111}
 \Cos{G}(\G) \to \cM_G(\G)
  \end{equation}
is a homeomorphism over the locus of irreducible representations (those that have no $G$-invariant 
proper subspaces $W\subset\CC^r$). 
If $\rho: \Gamma \to G \subset \GL_r(\CC)$ is reducible, then it has a (maximal) filtration $W_0=0\subsetneq W_1\subsetneq \ldots \subsetneq W_m=\CC^r$,
such that the induced representations $\rho_k$ on $W_k/W_{k-1}$, $k=1,\ldots,m$, are irreducible. 
We call $\Gr(\rho)=\rho_1\oplus \ldots \oplus \rho_m$ the semi-simplification of $\rho$ and we say that
$\rho,\rho'$ are S-equivalent if they have the same semi-simplification. With all this said, 
the fibers of (\ref{eqn:111}) are the S-equivalence classes \cite[Theorem 1.28]{LuMa}. 

On the other hand, fixed an element $\gamma\in \G$, we define the associated \emph{character} as the map $\chi_\gamma :\Rep{G}(\G) \longrightarrow \CC$ given by $\chi_\gamma(\rho) = \tr \rho(\gamma)$. 
This defines a $G$-invariant function. The \emph{character variety} is the algebraic space defined
by these functions,
$$
 \chi_{G}(\G)=\Spec \CC[\chi_\gamma \, | \, \gamma\in \G].
 $$
By the results of \cite{Lawton-Sikora:2017} and \cite[Chapter 1]{LuMa}, for $G = \SL_n(\CC), \mathrm{Sp}_{2n}(\CC)$ or $\mathrm{SO}_{2n+1}(\CC)$ this is isomorphic to $\cM_G(\G)$.

The main focus of this paper are the representation varieties for \emph{surface groups}.
Let $\Sigma_{g}$ be a compact orientable surface of genus $g$. Its fundamental group is 
 \begin{equation}\label{eqn:piSg}
  \G=\pi_{1}(\Sigma_{g})= \Big\langle a_1,b_1,\ldots, a_g,b_g \, \Big| \, \prod_{j=1}^g [a_j,b_j]=1 \Big\rangle.
  \end{equation}
The representation variety over the surface group $\pi_{1}(\Sigma_{g})$, 
denoted by $\Rep{G}(\Sigma_{g})$, parametrizes local systems over $\Sigma_g$.
For $G=\GL_r(\CC)$, the variety $\Rep{G}(\Sigma_{g})\sslash G$ is also known as the 
Betti moduli space in the context of non-abelian Hodge theory.
Let $K=\UU(r)$ be the maximal compact subgroup of $G=\GL_r(\CC)$.
The celebrated theorem by Narasimhan and Seshadri in \cite{NS} establishes that if we give $\Sigma_g$ a complex structure, 
then $\Rep{\UU(r)}^{ss}(\Sigma_{g})\slash \UU(r)$ is isomorphic to the moduli space of (polystable) 
holomorphic bundles of degree $0$ on $\Sigma_{g}$, where
$\Rep{\UU(r)}^{ss}(\Sigma_{g})$ are the semi-simple representations. 
The Narasimhan-Seshadri correspondence can be considered an extension to higher ranks of the classical Hodge theorem. 
A representation $\rho: \pi_{1}(\Sigma_{g}) \to \UU(1)$ can be regarded as a cohomology class $[\rho]\in H^{1}(\Sigma_{g},\CC)$. 
Indeed, the $\Rep{\UU(1)}(\Sigma_{1})$ is isomorphic to 
 $$
 \text{Hom}(\pi_{1}(\Sigma_{g})/[\pi_{1}(\Sigma_{g}),\pi_{1}(\Sigma_{g})],\UU(1)) \cong \text{Hom}(H_{1}(\Sigma_{g}),\CC)\cong H^{1}(\Sigma_{g},\CC),
 $$
because $\UU(1)$ is abelian. 
The classical Hodge theorem then says that there is a decomposition $\rho=\eta\oplus \omega$ where $\eta\in H^{0,1}(\Sigma_{g})$ and $\omega\in H^{1,0}(\Sigma_{g})$. Therefore $\eta$ provides us with a holomorphic line bundle, 
that is, an holomorphic object reflecting the algebraic structure of $\Sigma_{g}$.

In general, for a complex reductive group $G$, $\cM_{G}(\Sigma_{g}) = \Rep{G}(\Sigma_{g})\sslash G$ is a hyperk\"{a}hler manifold, that is a manifold, modelled on the quaternions, with three complex structures $I,\,J\,$ and $K$, where $I$ is the complex structure inherited from the complex structure of the group $G$, in the same fashion as shown in section \ref{subsec:charvar}, $J$ is the complex structure provided by the complex structure of $\Sigma_{g}$ as explained above, and $K$ is the product $JI$. Therefore by focusing on only one of the complex structures, three moduli spaces are obtained: the moduli space $\cM_{G}(\Sigma_{g})$ of representations of the fundamental group of $\Sigma_{g}$ into $G$ for complex structure $I$, also known as Betti moduli space; the moduli space of polystable $G$-Higgs bundles of degree $0$ on $\Sigma_{g}$ for complex structure $J$, called Dolbeault moduli space; and the moduli space of polystable flat bundles on $\Sigma_{g}$ with vanishing first Chern  class, known as the de Rham moduli space.
Moreover, the work of Corlette, Donaldson, Hitchin and Simpson (see \cite{Corlette:1988, Donaldson1, Hitchin, Simpson:1992,SimpsonI,SimpsonII}) proves that there are diffeomorphisms between the three moduli spaces: Betti, Dolbeault and de Rham.  These diffeomorphisms expand the Riemann-Hilbert correspondence and Narasimhan-Seshadri theorem into what is known as the \emph{non-abelian Hodge correspondence}.

The diffeomorphism between $\cM_{G}(\Sigma_{g})$ and the Dolbeault moduli space has been largely exploited to obtain information on the topology of the character variety since Hitchin's work in \cite{Hitchin}. Moreover, the rich interaction between string theory and the moduli space of $G$-Higgs bundles has driven the most recent research on character varieties. There exists a map, known as the Hitchin map, that shows the moduli space of Higgs bundles as a fibration over a vector space. This fibration was proved by Hausel and Thaddeus in \cite{hausel-thaddeus:2003} to be the first non-trivial example of mirror symmetry, following Strominger, Yau and Zaslow's definition in \cite{Strominger-Yau-Zaslow}. That is, for Langlands dual groups $G$ and $^L G$, the Hitchin map fibres over the same vector space in such a way that the fibres for the $G$-Higgs bundles moduli space are dual Calabi-Yau manifolds to the fibres of the Hitchin map for $^L G$-Higgs bundles moduli space. In order to prove so, Hausel and Thaddeus studied the Hodge numbers for these moduli spaces. Since our non-abelian Hodge correspondence is not an algebraic isomorphism it leads to one of the many motivations to  study the Hodge numbers for character varieties. 
We  introduce the Hodge numbers in Section \ref{subsec:Hodge}.

This discussion is at the heart of much recent research that justifies the study of the geometry of character varieties of surface groups,
in particular the Hodge numbers and $E$-polynomials (defined in Section \ref{subsec:Hodge}), since they are algebro-geometric invariants
associated to the complex structure. The first technique for this was the \emph{arithmetic method} inspired in the Weil conjectures.
Hausel and Rodr\'iguez-Villegas started the computation of the $E$-polynomials of $G$-character
varieties of surface groups for $G=\GL_n(\CC)$, $\SL_n(\CC)$ and $\PGL_n(\CC)$, using arithmetic methods. 
In \cite{Hausel-Rodriguez-Villegas:2008} they obtained the $E$-polynomials of the Betti moduli spaces 
for $G=\GL_n(\CC)$ 
in terms of a simple generating function.
Following these methods, Mereb \cite{mereb} studied this case for $\SL_n(\CC)$, giving
an explicit formula for the $E$-polynomial in the case $G=\SL_2(\CC)$.
Recently, using this technique, explicit expressions of the $E$-polynomials have been computed \cite {Baraglia-Hekmati:2016}
for orientable surfaces with $G= \GL_3(\CC)$, $\SL_3(\CC)$ and for non-orientable surfaces with $G= \GL_2(\CC)$, $\SL_2(\CC)$.  

A \emph{geometric method} to compute $E$-polynomials of character varieties of surfaces groups was
initiated by Logares, Mu\~noz and Newstead in \cite{LMN}. In this method, the representation variety is chopped into 
simpler strata for which the $E$-polynomial can be computed. Following this idea, in the case $G=\SL_2(\CC)$ 
the $E$-polynomials were computed in a series of papers \cite{LMN,MM:2016,MM} and for
$G=\PGL_2(\CC)$ in \cite{Martinez:2017}. This method yields all the polynomials explicitly, and not in terms of generating functions.
Moreover it allows to keep track of interesting properties, like the Hodge-Tate condition (c.f.\ Remark \ref{rem:22}) of these spaces.

In the papers \cite{LMN, MM}, the authors show that a recursive pattern underlies the computations. The 
$E$-polynomial of the $\SL_2(\CC)$-representation variety of $\Sigma_g$ can be obtained from some data of the representation 
variety on the genus $g-1$ surface. The recursive nature of character varieties is widely present in the literature as in \cite{Diaconescu:2017,Hausel-Letellier-Villegas:2013}. It suggests that some type of recursion formalism, in the spirit of a Topological Quantum Field Theory (TQFT for short), must hold. This leads to the third computational method, the \emph{quantum method}, introduced
in \cite{GPLM-2017}, that formalizes this set up and provides a powerful machinery to compute $E$-polynomials of
character varieties. Moreover, this technique allows us to keep track of the classes in the Grothendieck ring of varieties (also known as virtual classes, as defined in section \ref{subsec:Groth}) of the representation varieties and had been successfully used in \cite{GP-2018a, GP-2019} in the parabolic context, in which we deal with punctured surfaces with prescribed monodromy around the puctures.

This paper applies the geometric, arithmetic and quantum methods to the group of affine transformation of the line,
$G=\Aff{1}{\CC}$. The representations of this group parametrize (flat) rank one affine bundles $L\to \Sigma_g$, so it is 
a relevant space per se. Moreover, despite of its simplicity, $G$ is not a reductive group, so the coincidence between the Betti moduli space and the character variety is not granted by \cite{Culler-Shalen}. Nonetheless, we will directly prove in section \ref{sec:moduli-character} that this isomorphism still holds. We shall see
how the three methods apply, performing explicit computations of their virtual classes. In this way, our main result is:
 \begin{theorem}\label{theorem:main}
Let $G=\Aff{1}{\CC}$ and $g\geq 1$. The virtual class for the representation variety $\Rep{\Aff{1}{\CC}}(\Sigma_g)$ is
\begin{eqnarray*}
&&[\Rep{\Aff{1}{\CC}}(\Sigma_{g})]= q^{2g-1} (q-1)^{2g} +q^{2g}-q^{2g-1} \, .
\end{eqnarray*}
\end{theorem}

\noindent {\bf Acknowledgments.} The authors want to thank Jesse Vogel for the very careful reading of this manuscript and for pointing out a mistake in the computation of section \ref{sec:moduli-character}, and to Sean Lawton for references. The third author is partially supported by Project MINECO (Spain) PGC2018-095448-B-I00.

\section{General Background}\label{sec:2}

\subsection{Character varieties}\label{subsec:charvar}

Let $\G$ be a finitely generated group and $G$ an algebraic group over a ground field $\bk$. 
A representation of $\G$ into $G$ is a group homomorphism 
$$
\rho:\G\longrightarrow G.
$$

We shall denote the set of representations $\Hom(\G,G)$, by $\Rep{G}(\G)$. Since $G$ is algebraic and $\G$ finitely presented, $\Rep{G}(\G)$ inherits the structure of an algebraic variety. Indeed, if we consider a presentation $\G=\langle\gamma_{1},\ldots, \gamma_{N}\,|\, R_{j}(\gamma_{1},\ldots, \gamma_{N})\rangle$ then the homomorphism
$$
\varphi:\Rep{G}(\G)\longrightarrow G^{N},\qquad \rho\mapsto (\rho(\gamma_{1}),\ldots,\rho(\gamma_{N})),
$$
describes an injection such that 
$$
\varphi(\Rep{G}(\G))= \big\{(g_{1},\ldots, g_{N})\in G^{N}\, \big|\, R_{j}(g_{1},\ldots, g_{N}) \big\},
$$
so that $\varphi(\Rep{G}(\G))$ is an affine algebraic variety.

The group $G$ itself acts on $\Rep{G}(\G)$ by conjugation, that is $g \cdot \rho(\gamma) = g\rho(\gamma)g^{-1}$ for any $g \in G$, $\rho \in \Rep{G}(\G)$ and $\gamma \in \G$. We are interested on the orbits by this action since two representations are isomorphic if and only if they lie in the same orbit. But parametrizing these orbits requires the use of a subtler technique known as Geometric Invariant Theory (GIT). Let us explain this in some detail.

\begin{example}\label{ex:1}
Consider the simplest case where $\G=\ZZ$ and let $G=\SL_{2}(\CC)$. Then $\Rep{\SL_2(\CC)}(\ZZ) = \SL_2(\CC)$. The quotient $\SL_{2}(\CC)/\SL_{2}(\CC)$ contains the following orbits: if $g\in \SL_{2}(\CC)$ has two different eigenvalues $\lambda,\,\lambda^{-1}$ then the orbit of $g$ is a closed one dimensional space, namely the collection of matrices of trace $\lambda + \lambda^{-1}$. But in the case $\lambda=\lambda^{-1}=\pm 1$ we get a non-closed one dimensional orbit and an orbit which consist of a point, which are respectively
$$\left[\left( \begin{matrix} \pm 1 & 1 \\ 0 & \pm 1 \end{matrix} \right)\right],\quad 
\left\{\left( \begin{matrix}\pm 1 & 0 \\ 0 & \pm 1 \end{matrix} \right)\right\}.
$$
Moreover, for all $t \neq 0$, we have that the matrices 
$$
\left( \begin{array}{cc} \pm 1 & t \\ 0 & \pm 1 \end{array} \right)\in
 \left[\left( \begin{array}{cc} \pm 1& 1 \\ 0 &\pm 1\end{array} \right)\right] ,
$$
but become the point orbit for $t=0$. Therefore $\SL_{2}(\CC)/\SL_{2}(\CC)$ is not an algebraic variety since its topology does not
satisfy the $T_1$ separation axiom. 
The GIT quotient  $\SL_{2}(\CC)\sslash \SL_{2}(\CC)$ solves this problem by collapsing the two $1$-dimensional open orbits with the 
two orbits consisting on just a point. In this way, $\SL_{2}(\CC)\sslash \SL_{2}(\CC) = \CC$.
\end{example}

In general, for any algebraic group $G$ acting on an affine variety $X$ over $\bk$, the action induces 
an action on the algebra of regular functions on $X$, $\cO(X)$. In this case, the affine GIT quotient is defined as the morphism 
$$
\varphi: X \longrightarrow X \sslash G :=\operatorname{Spec} \cO(X)^{G}
$$
of affine schemes associated to the inclusion $\varphi^{*}:\cO(X)^{G}\hookrightarrow \cO(X)$, where $ \cO(X)^G$ is the subalgebra of $G$-invariant functions.

 
\begin{remark}  \label{rem:GITquo}
In general, the GIT quotient $X \sslash G$ is only an affine scheme since $\cO(X)^{G}$ might not be finitely generated (for an example of this phenomenon, see \cite{Nagata:1960}). However, a theorem of Nagata \cite{Nagata:1963} shows that, if $G$ is a reductive group (c.f.\ \cite[Chapter 3]{Newstead:1978}), then $\cO(X)^G \subseteq \cO(X)$ is finitely generated subalgebra and, thus, $X \sslash G$ is an affine variety. Many typical algebraic groups are reductive like $\GL_r(\CC), \SL_r(\CC)$ or $\CC^*$ with multiplication. However, an easy example of a non-reductive group is $\CC$ with the sum. 
\end{remark}

The key point of the GIT quotient is that it is a quotient from a categorical point of view. 
A \emph{categorical quotient} for $X$ is a $G$-invariant regular map of algebraic varieties $\varphi: X \to Y$ 
such that for any $G$-invariant regular map of varieties $f: X \to Z$, there exists a unique $\tilde{f}: Y \to Z$ such that the following diagram commutes
	\[
\begin{displaystyle}
   \xymatrix
   { X \ar[d]_{\pi} \ar[r]^f& Z \\
     Y \ar@{--{>}}[ru]_{\tilde{f}} &
      }
\end{displaystyle}   
\]
Using this universal property, it can be shown that if a categorical quotient exists, it is unique up to regular isomorphism. In this sense, it is straightforward (c.f.\ \cite[Corollary 3.5.1]{Newstead:1978}) to check that the GIT quotient (if it is a variety, see Remark \ref{rem:GITquo}) 
is a categorical quotient. Thus, it is uniquely determined by this universal property.

\begin{example}
In Example \ref{ex:1}, we have that the trace $\tr : \SL_2(\CC) \longrightarrow \CC$ 
is the only non-trivial $\SL_{2}(\CC)$-invariant function on $\SL_2(\CC)$. Therefore $\SL_{2}(\CC)\sslash\SL_{2}(\CC) = \Spec \CC[\tr] = \CC$. In general rank $r > 1$, we have that $\SL_{r}(\CC)\sslash\SL_{r}(\CC) = \CC^{r-1}$ with quotient map given by the coefficients of the characteristic polynomial.
\end{example}

Coming back to our case of study, we have an action of $G$ on $\Rep{G}(\G)$ by conjugation. The GIT quotient is called the \emph{moduli space of representations} and it is denoted as
$$
	\cM_G(\G)=\Rep{G}(\G) \sslash G.
$$
By construction, there is a natural continuous map from the coset space $\Cos{G}(\G)$, that 
parametrizes the isomorphisms classes of representations of $\G$ into $G$, to this space $\Cos{G}(\G)\to \cM_G(\G)$.

However, if the ground ring is $\bk = \CC$ (or, in general, algebraically closed), 
we may consider another natural way of parametrize isomorphism classes of representations. 
Suppose that $G$ is a linear algebraic group, so that $G<\GL_r(\CC)$. Given a representation $\rho:\G\to G$ we define its character as the map 
$$
\chi_{\rho}:\G\longrightarrow \CC,\quad  \gamma \mapsto\chi_{\rho}(\gamma) =\tr \rho(\gamma).
$$ 

Note that two isomorphic representations $\rho$ and $\rho'$ have the same character, whereas the converse 
is also true if $\rho$ and $\rho'$ are \emph{irreducible} (see \cite[Proposition 1.5.2]{Culler-Shalen}). A representation
is irreducible is it has no proper $G$-invariant subspaces of $\CC^r$, otherwise it is called \emph{reducible}.

If $\rho$ is reducible, let $\CC^k\subset \CC^r$ be a proper $G$-invariant subspace. Define $\rho_1:=\rho|_{\CC^k}$, which
is a representation on $\CC^k$. There is an induced representation $\rho_2$ in the quotient $\CC^{r-k}=\CC^r/\CC^k$. Then, we can write
 $$
 \rho=\begin{pmatrix} \rho_1 & M \\ 0 &\rho_2\end{pmatrix}.
 $$
Acting by conjugation by matrices $\begin{pmatrix} t\Id  & 0 \\ 0 &\Id \end{pmatrix}$, we see that 
 $\rho$ is equivalent to $\rho_t=\begin{pmatrix} \rho_1 & tM \\ 0 &\rho_2\end{pmatrix}$. When taking $t\to 0$, we have
 that $\rho$ is in the same GIT orbit than $\begin{pmatrix} \rho_1 & 0 \\ 0 &\rho_2\end{pmatrix}=\rho_1\oplus \rho_2$.
This is the same situation of
Example \ref{ex:1}. Repeating the argument with $\rho_2$, we have that any $\rho$ is equivalent to some
$\rho_1\oplus\ldots\oplus \rho_l$, where $\rho_j$ are irreducible. This is called a \emph{semi-simple representation}. We say that they are S-equivalent,
and denote $\rho \sim \rho_1\oplus\ldots\oplus \rho_l$. In this way, any point of the GIT-quotient is determined by a unique class
of semi-simple representation.

There is a character map 
$$
\chi:\Rep{G}(\G)\longrightarrow \CC^{\Gamma}, \quad \rho\mapsto \chi_{\rho}
$$ 
whose image $\chi_G(\G) = \chi(\Rep{G}(\G))$ is called the \emph{$G$-character variety} of $\G$. 
Moreover, by the results in \cite{Culler-Shalen} there exist a collection $\gamma_{1},\ldots, \gamma_{a}$ of elements of $\G$ such that $\chi_{\rho}$ is 
determined by $(\chi_{\rho}(\gamma_{1}), \ldots, \chi_{\rho}(\gamma_{a}))$, for any $\rho$. Such collection gives a map 
$$
\phi:\Rep{G}(\G)\longrightarrow \CC^{a}, \qquad \phi(\rho)=(\chi_{\rho}(\gamma_1),\ldots \chi_\rho(\gamma_{a})),
$$
and we have a bijection  $\chi_{G}(\G) \cong \phi (\Rep{G}(\G))$ which endows $\chi_{G}(\G)$ with the structure of an algebraic variety independent from the collection  $\gamma_{1},\ldots, \gamma_{a}$ chosen. 

The character map $\chi: \Rep{G}(\G) \to \chi_G(\G)$ is a regular $G$-invariant map so, since the GIT quotient is a categorical quotient, it induces a map
$$
	\tilde{\chi}: \cM_G(\G) \to \chi_G(\G).
$$
It is well-known that, when the group $G = \SL_n(\CC)$, this map is an isomorphism \cite{Culler-Shalen}. This is the
reason for the fact that sometimes the space $\cM_G(\G)$ is called the character variety. For different groups this isomorphism may still hold, as in this paper for $G=\Aff{1}{\bk}$, or may not hold as in \cite[Appendix A]{Florentino-Lawton:2012} for $G=\textrm{SO}_2$. For a general discussion about the relation of $\cM_G(\G)$ and $\chi_G(\G)$, see \cite{Lawton-Sikora:2017}.

\subsection{Representation varieties of orientable surfaces} \label{subsec:repres}

A very important class of representation varieties appears when consider representations of the fundamental group of a compact surface, the so-called surface groups. Let $\Sigma_{g}$ be a compact orientable surface of genus $g$. We take $\G = \pi_1(\Sigma_g)$ and we will focus on the representation variety $\Rep{G}(\pi_1(\Sigma_g))$, that we will shorten as $\Rep{G}(\Sigma_g)$. Using the presentation (\ref{eqn:piSg}) of $\pi_1(\Sigma_g)$, we get that
$$
\Rep{G}(\Sigma_{g})= \Big\{(A_{1},B_1, \ldots, A_{g},B_g)\in G^{2g}\,\Big|\, \prod_{j=1}^g [A_{j},B_j] \Big\}\subset G^{2g}.
$$

The associated moduli space of representations, $\cM_G(\Sigma_g) = \Rep{G}(\Sigma_{g})\sslash G$, plays a fundamental role in the so-called non-abelian Hodge correspondence in the case $G = \GL_r(\CC)$ (resp.\ $G = \SL_r(\CC)$). To be precise, consider a complex vector bundle
$$
	\pi: E \to \Sigma_g
$$ of rank $r$ and degree $0$ (resp.\ and trivial determinant line bundle) with a flat connection $\nabla$ on $E$. By flatness, there is no local holonomy for $\nabla$, so the holonomy map does not depend on the homotopy class of the loop, hence it descends to a map, called the \emph{monodromy}
  $$
   \rho_\nabla :\pi_1(\Sigma_g) \to G.
  $$
This is a representation in $\Rep{G}(\Sigma_g)$. The isomorphism class of the pair $(E, \nabla)$ is given by changing the basis of the fiber $E_{x_0} =\CC^r$ over the base point $x_0\in \Sigma_g$. This produces the action by conjugation of $G$ on $\Rep{G}(\Sigma_g)$.

In this way, the moduli of representations $\cM_G(\Sigma_g)=\Rep{G}(\Sigma_g)\sslash G$ parametrizes the moduli space of classes of pairs $(E, \nabla)$ of flat connections on a vector bundle (modulo S-equivalence). In this context, the former space is usually referred to as the Betti moduli space (it captures topological information of $\Sigma_g$), and the later space that is called the de Rham moduli space (it captures differentiable information of $\Sigma_g$).

\subsection{Mixed Hodge structures} \label{subsec:Hodge}

In order to understand the geometry of representation varieties of surface groups, we will focus on an algebro-geometric 
invariant that is naturally present in the cohomology of complex varieties, the so-called \emph{Hodge structure}. 
For this reason, in this section, we will consider that the ground ring is $\CC$ 
and we will sketch briefly some remarkable properties of Hodge theory. For a more detailed introduction to Hodge theory, see \cite{Peters-Steenbrink:2008}.

A pure Hodge structure of weight $k$ consists of a finite dimensional rational vector space
$H$ whose complexification $H_\CC = H \otimes_\QQ \CC$ is equipped with a decomposition
$$
	H_\CC=\bigoplus\limits_{k=p+q} H^{p,q},
$$
such that $H^{q,p}=\overline{H^{p,q}}$, the bar meaning complex conjugation on $H$.
A Hodge structure of weight $k$ gives rise to the so-called Hodge filtration, which is a descending filtration
$F^{p}=\bigoplus\limits_{s\ge p}H^{s,k-s}$. From this filtration we can recover the pieces via the graded complex $\Gr^{p}_{F}(H):=F^{p}/ F^{p+1}=H^{p,k-p}$.

A mixed Hodge structure consists of a finite dimensional rational vector space $H$,
an ascending (weight) filtration $0 \subset \ldots \subset W_{k-1}\subset W_k \subset \ldots \subset H$ and a descending (Hodge) filtration $H_\CC \supset \ldots \supset F^{p-1}\supset F^p \supset \ldots \supset 0$ such that $F$ induces a pure Hodge structure
of weight $k$ on each $\Gr^{W}_{k}(H)=W_{k}/W_{k-1}$. We define the associated Hodge pieces as
 $$
 H^{p,q}:= \Gr^{p}_{F}\Gr^{W}_{p+q}(H)_\CC
 $$
and write $h^{p,q}$ for the {\em Hodge number} $h^{p,q} :=\dim_\CC H^{p,q}$. 

The importance of these mixed Hodge structures rises from the fact that the cohomology of complex algebraic varieties are naturally endowed with such structures, as proved by Deligne.

\begin{theorem}[Deligne \cite{DeligneI:1971,DeligneII:1971,DeligneIII:1971}]
Let $X$ be any quasi-projective complex algebraic variety (maybe non-smooth or non-compact). 
The rational cohomology groups $H^k(X)$ and the cohomology groups with compact support  
$H^k_c(X)$ are endowed with mixed Hodge structures. 
\end{theorem}

In this way, for any complex algebraic variety $X$, we define the {\em Hodge numbers} of $X$ by
 \begin{eqnarray*}
 h^{k,p,q}(X)&=&h^{p,q}(H^k(Z))=\dim \Gr^{p}_{F}\Gr^{W}_{p+q}H^{k}(X)_\CC ,\\
 h^{k,p,q}_{c}(X)&=&h^{p,q}(H_{c}^k(Z))=\dim \Gr^{p}_{F}\Gr^{W}_{p+q}H^{k}_{c}(X)_\CC.
 \end{eqnarray*}

The $E$-polynomial (also called Deligne-Hodge polynomial) is defined as 
 $$
 e(X)=e(X)(u,v):=\sum _{p,q,k} (-1)^{k}h^{k,p,q}_{c}(X) u^{p}v^{q}.
 $$

The key property of $E$-polynomials that permits their calculation is that they are additive for
stratifications of $X$. If $X$ is a complex algebraic variety and
$X=\bigsqcup\limits_{i=1}^{n}X_{i}$, where all $X_i$ are locally closed in $X$, then $e(X)=\sum\limits_{i=1}^{n}e(X_{i})$. Moreover, if $X = F \times B$, the K\"uneth isomorphism implies that $e(X) = e(F)e(B)$.

An easy consequence of these two properties is that, indeed, for an algebraic bundle (that is, locally trivial in the Zariski topology)
 $$
  F\longrightarrow X \stackrel{\pi}{\longrightarrow} B,
  $$
we have $e(X)=e(F)e(B)$. For this, just take a Zariski open subset $U\subset B$ so that $X|_U =\pi^{-1}(U) \cong U\x B$.
Then $B_1=B-U$ is closed and we can repeat the argument for $F\to X|_{B_1} \to B_1$. By the noethereanity,
we get a finite chain 
 $$ 
 B_{n+1}=\emptyset \subsetneq B_n \subsetneq \ldots \subsetneq B_1\subsetneq B=B_0,
 $$ 
where $U_k=B_{k-1}-B_k$ is Zariski open in $B_{k-1}$ and $X|_{U_k} \cong U_k\x B$. Then
 \begin{equation}\label{eqn:lll}
e(X)=\sum_{k} e(X|_{U_k}) = \sum_{k} e(F) e(U_k)=e(F) \sum_{k} e(U_k)= e(F) e(B).
 \end{equation}

\begin{example}
Recall that the cohomology of the complex projective space, $H^\bullet(\PP^n)$, is generated by the Fubini-Study 
form which is of type $(1,1)$, so we get $h_c^{2p,p,p}(\PP^n)=1$ for $0 \leq p \leq n$, and $0$ otherwise. Hence, its $E$-polynomial is $e(\PP^n) = 1 + uv + u^2v^2 + \ldots + u^nv^n$. In particular, since $\PP^1 = \CC \sqcup \left\{\infty\right\}$ we get that $e(\CC) = e(\PP^1) - 1 = uv$. In this way, we get that $e(\CC^n) = u^nv^n$, which is compatible with the usual decomposition $\PP^n = \star \sqcup \CC \sqcup \CC^2 \sqcup \ldots \sqcup \CC^n$.
\end{example}

\begin{remark} \label{rem:22}
When $h_c^{k,p,q}(X)=0$ for $p\neq q$, the polynomial $e(X)$ depends only on the product $uv$.
This will happen in all the cases that we shall investigate here. In this situation, it is
conventional to use the variable $q=uv$. If this happens, we say that the variety is of Hodge-Tate type (also known as balanced type). For instance, $e(\CC^{n})=q^{n}$ is Hodge-Tate. 
\end{remark}

\subsection{Grothendieck ring of algebraic varieties} \label{subsec:Groth}

Recall that from a (skeletally small) abelian category $\cA$, 
it is possible to construct an abelian group, known as the Grothendieck group of $\cA$. It is the abelian group $\K{\cA}$
generated by the isomorphism classes $[A]$ of objects $A\in \cA$, subject to the relations that whenever there exists a short exact sequence 
$
0\to B\to A\to C\to 0
$
we declare $[A]=[B]+[C]$.
Furthermore, if our abelian category is provided with a tensor product, i.e. $\cA$ is 
monoidal, and the functors $-\otimes A:\cA\rightarrow \cA$ and $A\otimes -: \cA\rightarrow \cA$ are exact, then $\K{\cA}$ inherits a ring structure by $[A]\cdot[B]= [A\otimes B]$ (see \cite{Weibel}), under which it is called the \emph{Grothendieck ring} of $\cA$. The elements $[A] \in \K{\cA}$ are usually referred to as \emph{virtual classes}.

In our case, we are interested on the category of algebraic varieties with regular morphisms $\Var{\bk}$ 
over a base field $\bk$, which is not an abelian category. Nevertheless, we can still construct its Grothendieck group, $\K{\Var{\bk}}$, in an analogous manner, that is, as the abelian group generated by isomorphism classes of algebraic varieties with the relation that $[X]=[Y]+[U]$ if $X=Y\sqcup U$, with $Y\subset X$ a closed subvariety. Furthermore, the cartesian product of varieties also provides $\K{\Var{\bk}}$ with a ring structure. A very important element 
is the class of the affine line, $q = [\bk] \in \K{\Var{\bk}}$, the so-called \emph{Lefschetz motive}. 

\begin{remark} \label{rem:zerodivisor}
Despite the simplicity of its definition, the ring structure of $\K{\Var{\bk}}$ is widely unknown. In particular, for almost fifty years it was an open problem whether it is an integral domain. Indeed, the answer is no and, more strikingly, the Lefschetz motive $q$ is a zero divisor \cite{Borisov:2014}.
\end{remark}

Observe that, due to its additivity and multiplicativity properties, the $E$-polynomial defines a ring homomorphism
$$
	e: \K{\Var{\CC}} \to \ZZ[u^{\pm 1}, v^{\pm 1}].
$$
This homomorphism factorizes through mixed Hodge structures. To be precise, Deligne proved in \cite{DeligneI:1971} that the category of mixed Hodge structures ${\MHS{\QQ}}$ is an abelian category. Therefore we may as well consider its Grothendieck group, $\K{\MHS{\QQ}}$, which again inherits a ring structure. The long exact sequence in cohomology with compact support and the K\"unneth isomorphism shows that there exists ring homomorphisms $\K{\Var{\CC}} \to \K{\MHS{\QQ}}$ given by $[X] \mapsto [H_c^\bullet(X)]$, as well as $\K{\MHS{\QQ}} \to \ZZ[u^{\pm 1}, v^{\pm 1}]$ given by $[H] \mapsto \sum h^{p,q}(H) u^{p}v^{q}$ such that the following diagram commutes
	\[
\begin{displaystyle}
   \xymatrix
   {	\K{\Var{\CC}} \ar[rd]_{e} \ar[r] & \K{\MHS{\QQ}} \ar[d] \\
   & \ZZ[u^{\pm 1}, v^{\pm 1}]
      }
\end{displaystyle}   
\]

\begin{remark}\label{remark:$E$-pol-lefschetz}
From the previous diagram, we get that the $E$-polynomial of the affine line is $q=e([\CC])$ which justifies denoting by $q=[\CC]\in\K{\Var{\CC}}$ the Lefschetz motive. This implies that if the virtual class of a variety lies in the subring of $\K{\Var{\bk}}$ generated by the affine line, then the $E$-polynomial of the variety coincides with the virtual class, seeing $q$ as a variable. This will have deep implications, as we will explore in the arithmetic method in 
Section \ref{sec:4}.
\end{remark}

\begin{example}
As for $E$-polynomials, proceeding as in (\ref{eqn:lll}), we can show that if $F\to E \to B$ is an algebraic bundle, then $[E]=[F]\cdot [B]$ in $\K\Var{\bk}$. This enables multiple computations. For instance, consider the fibration $\CC \to \SL_{2}(\CC) \to \CC^{2}-\{(0,0)\}$, $f \mapsto f(1,0)$.
It is locally trivial in the Zariski topology, and therefore $[\SL_2(\CC)]=[\CC]\cdot[\CC^{2}-\{(0,0)\}] =q(q^2-1)=q^3-q.$

It is of interest to notice that one can compute $e(\PGL_{2}(\CC))=e(\SL_{2}(\CC))$, which is of no surprise since these groups are Langlands dual.
\end{example}

\section{Geometric method} \label{sec:3}

Using the previous machinery, let us show in a simple situation how to compute the virtual classes of representation varieties for surface groups. We will do this computation by three different approaches, the so called geometric, arithmetic and quantum method. The first geometric method, that we will follow in this section, is based on giving an explicit expression of the representation variety and chopping it into simpler pieces to ensemble the total virtual class. This is the method used in \cite{LMN,MM,MM:2016} to compute the $\SL_2(\CC)$-character varieties of surface groups. In Section \ref{sec:4}, 
we shall use the arithmetic methods of \cite{Hausel-Rodriguez-Villegas:2008}, 
based on counting the number of points of the representation variety over finite fields. 
Finally, in section \ref{sec:5} we shall use the machinery of the Topological Quantum Field Theories developed in \cite{GPLM-2017} 
to offer an alternative approach.

Let $\Sigma_g$ be the closed oriented surface of genus $g\geq 1$ as before. As target group we fix $G = \Aff{1}{\bk}$, the group of $\bk$-linear affine transformations of the affine line. Its elements are the matrices of the form
$
\begin{pmatrix}
	a & b\\
	0 & 1\\
\end{pmatrix}
$,
with $a \in \bk^* = \bk - \set{0}$ and $b \in \bk$. The group operation is given by matrix multiplication. In this way, $\Aff{1}{\bk}$ is isomorphic to the semidirect product $\bk^* \ltimes_\varphi \bk$ with the action $\varphi: \bk^* \times \bk \to \bk$, $\varphi(a,b)=ab$.

The representation variety is given by
$$
\Rep{\Aff{1}{\bk}}(\Sigma_{g})= \Big\{(A_{1},A_2, \ldots, A_{2g-1},A_{2g})\in \Aff{1}{\bk}^{2g}\,\Big|\, \prod_{i=1}^g [A_{2i-1},A_{2i}]= \I  \Big\}.
$$
Therefore, if we write
$$
A_i = \begin{pmatrix}
	a_i & b_i\\
	0 & 1\\
\end{pmatrix},
$$
then the product of commutators is given by 
 \begin{equation}\label{eqn:commutator}
	\prod_{i=1}^g \left[\begin{pmatrix}
	a_{2i-1} & b_{2i-1} \\
	0 & 1\\
\end{pmatrix},\begin{pmatrix}
	a_{2i} & b_{2i} \\
	0 & 1\\
\end{pmatrix}\right] = \begin{pmatrix}
	1 \, \, & {\displaystyle \sum_{i=1}^g (a_{2i-1}-1)b_{2i} - (a_{2i}-1)b_{2i-1} }\\
	0 \, \, & 1\\
\end{pmatrix}.
 \end{equation}

We can identify this variety with a more familiar space. Consider the auxiliary variety
 \begin{equation}\label{eqn:commutator2}
	X_{s} = \Big\{(\alpha_1, \ldots, \alpha_{s}, \beta_1, \ldots, \beta_{s}) \in (\bk-\{-1\})^s \times \bk^s\,\Big|\, \sum_{i=1}^{s} \alpha_i\beta_i = 0\Big\},
 \end{equation}
so that 
 $$
  \Rep{\Aff{1}{\bk}}(\Sigma_g) \cong X_{2g}
  $$ 
via the morphism $(a_{2i-1}, b_{2i-1}, a_{2i}, b_{2i}) \mapsto (a_{2i-1}-1, a_{2i}-1, \ldots,  b_{2i}, -b_{2i-1})$. Take $U = (\bk-\{-1\})^s - \left\{(0, \ldots, 0)\right\}$ and $V = U \times \bk^s$. We have that $X_s|_V$ is the pullback of the total space of the hyperplane bundle on $\PP^{s-1}$, $\cO_{\PP^{s-1}}(1)$, via the natural quotient map $\pi: U \subset \bk^s-\left\{0\right\} \to \PP^{s-1}$. That is, we have a pullback
	\[
\begin{displaystyle}
   \xymatrix
   {	X_s|_V = \pi^*\cO_{\PP^{s-1}}(1) \ar[d] \ar[r] & \cO_{\PP^{s-1}}(1) \ar[d] \\ U \ar[r]_\pi & \PP^{s-1}
      }
\end{displaystyle}   
\]
On the special fiber, $X_s|_{\{(0, \ldots, 0)\} \times \bk^s} = \bk^s$, which corresponds to the natural completion of the total space of the hyperplane bundle to the origin.
 
\subsection{Stratification analysis and computation of virtual classes}
  
Using this explicit description, we can compute the virtual class of the representation variety in a geometric way, by chopping the variety into simpler pieces, as shown in the following result. 
 
 \begin{theorem} \label{thm:44}
 The virtual class in the Grothendieck ring of algebraic varieties of the representation variety is
\begin{eqnarray*}
&&[\Rep{\Aff{1}{\bk}}(\Sigma_{g})]= q^{2g-1} \left((q-1)^{2g} + q-1\right) \, .
\end{eqnarray*}
\end{theorem}

\begin{proof}
We stratify the varieties $X_s$ in the following manner:
\begin{align*}
	X_s &= \Big\{\beta_s = \frac{1}{\alpha_s}\sum\limits_{i=1}^{s-1} \alpha_i\beta_i, \alpha_s \neq 0 \Big\} \bigsqcup 
	\Big\{\sum_{i=1}^{s-1} \alpha_i\beta_i=0,	\alpha_s=0\Big\} \\
	&= \left(\left((\bk-\{-1\}) \times \bk\right)^{s-1} \times (\bk-\set{0,-1}) \right) \sqcup \left(X_{s-1} \times \bk\right).
\end{align*}
This gives rise to the recursive formula for the virtual classes
 $$
 [X_s] = (q-2)q^{s-1}(q-1)^{s-1} + q[X_{s-1}].
 $$
The base case is 
 $$
  X_1 = \{(\alpha, \beta) | \alpha \beta=0\} =
  \left\{\beta= \frac{1}{\alpha}, \alpha \neq 0, -1\right\} \sqcup \left\{(0, \beta) \right\} \\
	= (\bk-\set{0,-1}) \sqcup \bk,
	$$ 
which has $[X_1] = 2q-2$. 
The induction gives
\begin{align*}
 [X_s] &= \sum_{t=1}^{s-1} (q-2)q^{s-t}(q-1)^{s-t} q^{t-1} + q^{s-1}(2q-2) \\
 &=(q-2)q^{s-1}  \frac{(q-1)^{s}-(q-1)}{(q-1)-1}+ 2q^{s-1}(q-1) \\
 &= q^{s-1} \big((q-1)^{s}-(q-1)\big)+ 2q^{s-1}(q-1) \\
 &=q^{s-1} (q-1)^{s} + q^s-q^{s-1} . 
\end{align*}

The representation variety is $\Rep{\Aff{1}{\bk}}(\Sigma_{g})\cong X_{2g}$, hence the result.
\end{proof}

\begin{remark}
In the case that $\bk = \CC$, the same formula of Theorem \ref{thm:44} 
gives the $E$-polynomial of the representation variety by seeing $q$ as a formal variable.
\end{remark}

\subsection{The moduli space of the representations and the character variety}\label{sec:moduli-character}

In this section, we will deal with the moduli space of representations, that is, the GIT quotient
 $$
 \cM_{\Aff{1}{\bk}}(\Sigma_{g})= \Rep{\Aff{1}{\bk}}(\Sigma_{1})\sslash \Aff{1}{\bk}.
  $$

For that purpose, let us write down the action explicitly. Consider elements
$$
P = \begin{pmatrix}
	\lambda & \mu\\
	0 & 1\\
\end{pmatrix} \in \Aff{1}{\bk}, \quad
\rho =\left( \begin{pmatrix}
	a_{1} & b_{1} \\
	0 & 1\\
\end{pmatrix},\ldots,\begin{pmatrix}
	a_{2g} & b_{2g} \\
	0 & 1\\
\end{pmatrix}
\right) \in \Rep{\Aff{1}{\bk}}(\Sigma_g)
$$
then we have that
$$
P \rho P^{-1} =\left( \begin{pmatrix}
	a_{1} & \lambda b_{1}+\mu(a_1-1) \\
	0 & 1\\
\end{pmatrix},\ldots,\begin{pmatrix}
	a_{2g} & \lambda b_{2g} + \mu(a_{2g}-1) \\
	0 & 1\\
\end{pmatrix}
\right).
$$

\begin{remark}
This action can be also understood in terms of $X_{2g}$. In this coordinates, the action of
$(\lambda, \mu) \in \bk^* \ltimes_\varphi \bk = \Aff{1}{\bk}$ is given by
\begin{align*}
  (\lambda, \mu) & \cdot (\alpha_1, \ldots, \alpha_{2g}, \beta_1,\ldots, \beta_{2g}) =\\
  &= (\alpha_1, \ldots, \alpha_{2g}, \lambda\beta_1+\mu \alpha_2, \lambda\beta_2-\mu \alpha_1,
  \ldots, \lambda\beta_{2g-1}+\mu \alpha_{2g}, \lambda\beta_{2g}-\mu \alpha_{2g-1}).
 \end{align*}
 \end{remark}
 
 In particular, if we take $\mu = 0$ we have that the action is given by
$$
P \rho P^{-1} =\left( \begin{pmatrix}
	a_{1} & \lambda b_{1} \\
	0 & 1\\
\end{pmatrix},\ldots,\begin{pmatrix}
	a_{2g} & \lambda b_{2g} \\
	0 & 1\\
\end{pmatrix}
\right) \stackrel{\lambda \to 0}{\longrightarrow} \left( \begin{pmatrix}
	a_{1} & 0 \\
	0 & 1\\
\end{pmatrix},\ldots,\begin{pmatrix}
	a_{2g} & 0 \\
	0 & 1\\
\end{pmatrix}\right).
$$
Therefore, any representation is S-equivalent to a diagonal representation, which implies that
$$
	\cM_{\Aff{1}{\bk}}(\Sigma_{g})= \Rep{\Aff{1}{\bk}}(\Sigma_{1})\sslash \Aff{1}{\bk} = (\bk^*)^{2g},
$$
so we get that $[\cM_{\Aff{1}{\bk}}(\Sigma_{g})] = (q-1)^{2g}$.

On the other hand, we also have the character variety $\chi_{\Aff{1}{\bk}}(\Sigma_{g})$ generated by the characters of the representations, as described in Section \ref{subsec:charvar}. Observe that given
$$
	\rho = (\rho(\gamma_1), \ldots, \rho(\gamma_{2g})) = \left( \begin{pmatrix}
	a_{1} & b_{1} \\
	0 & 1\\
\end{pmatrix},\ldots,\begin{pmatrix}
	a_{2g} & b_{2g} \\
	0 & 1\\
\end{pmatrix}
\right) \in \Rep{\Aff{1}{\bk}}(\Sigma_g),
$$
where $\gamma_1, \ldots, \gamma_{2g}$ are the standard generators of $\pi_1(\Sigma_g)$, its character is determined by the tuple
$$
	(\rho(\gamma_1), \ldots, \rho(\gamma_{2g})) = (a_1+1, \ldots, a_{2g}+1) \in (\bk-\{1\})^{2g}.
$$
Reciprocally, any tuple of $(\bk-\{1\})^{2g}$ is the character of an $\Aff{1}{\bk}$-representation, namely, the diagonal one. Hence, we have that 
 $$
 \chi_{\Aff{1}{\bk}}(\Sigma_{g}) = \phi(\Rep{\Aff{1}{\bk}}) = (\bk-\{1\})^{2g}\, .
 $$ 

In particular, this shows that $[\chi_{\Aff{1}{\bk}}(\Sigma_{g})] = (q-1)^{2g}$. Observe that we indeed
have an isomorphism $\cM_{\Aff{1}{\bk}}(\Sigma_{g}) \cong \chi_{\Aff{1}{\bk}}(\Sigma_{g})$ given by
$(a_1, \ldots, a_{2g}) \mapsto (a_1 + 1, \ldots, a_{2g} + 1)$. Notice that this isomorphism is not directly provided
by \cite{Culler-Shalen}.

\section{Arithmetic method}\label{sec:4}

In this section, we explore a different approach to the computation of $E$-poly\-nomials with an arithmetic flavour. This approach was initiated with the works of Hausel and Rodr\'iguez-Villegas \cite{Hausel-Rodriguez-Villegas:2008}. The key idea is based on a theorem of Katz that, roughly speaking, states that if the number of points of a variety $X$ over the finite field of $q$ elements, is a polynomial in $q$, $P(q) = |X(\FF_q)|$, then the $E$-polynomial of $X(\CC)$ is also $P(q)$. Under this point of view, the computation of $E$-polynomials reduces to the arithmetic problem of counting points over finite fields.

\subsection{Katz theorem and $E$-polynomials}\label{subsec:Katz}

Let us explain the result proved in \cite[Appendix]{Hausel-Rodriguez-Villegas:2008}. Start with a scheme $X/\CC$ over $\CC$. 
Let $R$ be a subring of $\CC$ which is finitely generated as a $\ZZ$-algebra and let $\cX$ be a separated $R$-scheme
of finite type. We call $\cX$
a \emph{spreading out} of $X$ if it yields $X$ after extension of scalars from $R$ to $\CC$.

We say that $\cX$ is \emph{strongly polynomial count} if there exists a
polynomial $P_\cX(T) \in \CC [T]$ such that for any finite field $\FF_q$ and any ring homomorphism $\varphi:R \to \FF_q$,
the $\FF_q$-scheme $\cX^\varphi$ obtained from $\cX$ by base change satisfies that 
for every finite extension $\FF_{q^n}/\FF_q$, we have
 $$
  \#\cX^\varphi(\FF_{q^n} ) = P_\cX(q^n).
  $$
We say that a scheme $X/\CC$ is \emph{polynomial count} if it admits a spreading out $\cX$ which is strongly polynomial count.

The following theorem is due to Katz \cite[Appendix]{Hausel-Rodriguez-Villegas:2008}. It computes the 
$E$-polynomial of $X$ from the count of
points of a spreading $\cX$.

\begin{theorem}
 Assume that $X$ is polynomial count with counting polynomial $P_\cX(T) \in \CC[T]$. Then
  $$
  e(X)=P_\cX (q),
  $$
where $q=uv$.
\end{theorem}

This is a powerful result that computes $E$-polynomials of varieties via arithmetic. For instance, it explains
easily the equality $e(X)=e(U)+e(Y)$, when $Y\subset X$ is a closed subset and $U=X-Y$ is the (open) complement.
Certainly, in this case 
 $$
 \#\cX^\varphi(\FF_{q^n} ) =\Big(\#\cY^\varphi(\FF_{q^n} ) \Big)+\Big(\#\cU^\varphi(\FF_{q^n} ) \Big),
 $$
for spreadings $\cX,\cY,\cU$ of $X,Y,Z$, respectively. Therefore $P_\cX(T)=P_\cY(T)+P_\cU(T)$, because they coincide
on a infinity of values $T=q^n$.
Note in particular that if $\cY,\cU$ are strongly polynomial count then $\cX$ is also strongly polynomial count.
This also implies that  the polynomial count only depends on the class in the Grothendieck ring.

The drawback of the arithmetic method is that it does not give information on the finer algebraic structure of the 
(mixed) Hodge polynomials, or the classes in the Grothendieck ring of varieties. For instance, 
the $E$-polynomial of an elliptic curve $X$ is $e(X) = 1 - u -v + uv$, which is not a polynomial in $q = uv$, and thus, $X$ cannot be polynomial count.

 \begin{corollary}
 Suppose that $X$ has class in the Grothendieck ring $[X]=P(q)$, where $P$ is a polynomial in the 
 Lefschetz motive $q=[\CC]$. Then $X$ is polynomial count with $e(X)=P(q)$, $q=uv$.
 \end{corollary}
 
 \begin{proof}
 As the statement only depends on the class in the Grothendieck ring, it is enough to prove it for $q^m$,
 that is $X=\CC^m$, for $m\geq 0$, where $P(T)=T^m$. The spreading for $X$ is given by $\cX=\Spec \ZZ[x_1,\ldots, x_m]$
 and $\cX^\varphi =\Spec \FF_q[x_1,\ldots, x_m]=\FF_{q}^m$. Therefore
 $\# \cX^\varphi(\FF_{q^n})=\# \FF_{q^n}^m=(q^n)^m=P(q^n)$. Hence $X$ is of polynomial
 count and its polynomial is $P(T)=T^m$. See also Remark \ref{remark:$E$-pol-lefschetz}.
 \end{proof}

In our situation, we start with an affine variety, which is of the form
 $$
  X=\Spec \frac{\CC[x_1,\ldots, x_N]}{I}\, ,
  $$
for some ideal $I=(p_1,\ldots, p_M)$, defined by
polynomials $p_1, \ldots, p_M \in \CC[x_1,\ldots, x_N]$. Take the coefficients of the polynomials, 
which are complex numbers, and let $R\subset \CC$ be the $\ZZ$-algebra generated by them. Then
$p_1,\ldots, p_M\in R[x_1,\ldots, x_N]$. A spreading of $X$ is given by
$$
 \cX=\Spec \frac{R[x_1,\ldots, x_N]}{(p_1,\ldots, p_M)}\, .
$$

A homomorphism $\varphi:R\to \FF_q$ defines polynomials $\bar p_j=\varphi(p_j)\in \FF_q[x_1,\ldots, x_N]$,
$j=1,\ldots,m$, and
$$
 \cX^\varphi=\Spec \frac{\FF_q[x_1,\ldots, x_N]}{(\bar p_1,\ldots, \bar p_M)}\, .
$$
This variety is 
 $$
  \cX^\varphi=V( \bar p_1,\ldots, \bar p_M) \subset \FF_q^N \, ,
  $$
and the $\FF_{q^n}$-points of $\cX^\varphi$ are the solutions over $\FF_{q^n}$ to the equations:
 $$
  \bar p_1(x_1,\ldots,x_N)= 0, \ldots,
  \bar p_M(x_1,\ldots,x_N)= 0.
  $$

\subsection{Representation variety for the affine group} \label{subsec:affine}

Let us take $G = \Aff{1}{\CC}$, the group of $\CC$-linear affine transformations of the complex line. 
As mentioned before, 
the character variety is $\Rep{\Aff{1}{\CC}}(\Sigma_g) \cong X_{2g}$, where
$$
	X_{s} = \Big\{(\alpha_1, \ldots, \alpha_{s}, \beta_1, \ldots, \beta_{s}) \in (\CC-\{-1\})^s \times \CC^s\,
	\Big|\, \sum_{i=1}^{s} \alpha_i\beta_i = 0\Big\}.
$$

The spreading of $X_s$ is given by taking the base-ring $R=\ZZ$ and the $\ZZ$-variety defined by
 $$
 \cX_{s} = \Spec \frac{\ZZ [\alpha_1, (\alpha_1+1)^{-1}, \ldots, \alpha_{s}, (\alpha_s+1)^{-1}, \beta_1, \ldots, \beta_{s}]}{\big(\sum_i \alpha_i\beta_i\big)}\, .
 $$
Take a prime $q$ and the quotient map $\varphi:\ZZ \to \ZZ_q=\FF_q$. 
This is followed by the embedding (scalar extension) $\FF_q\subset \FF_{q^n}$.
Hence
 $$
 \cX_s^\varphi(\FF_{q^n})= \Big\{(\alpha_1, \ldots, \alpha_{s}, \beta_1, \ldots, \beta_{s}) \in (\FF_{q^n}-\{-1\})^s \times \FF_{q^n}^s\,\Big|\, \sum_{i=1}^{s} \alpha_i\beta_i = 0\Big\},
$$
and we want to count the number of points.

\begin{theorem}
The variety $\cX_s$ is strongly polynomial count with polynomial $P_{\cX_s} (T)=T^{s-1}(T-1)^s+T^s-T^{s-1}$.
In particular, the $E$-polynomial of $\Rep{\Aff{1}{\CC}}(\Sigma_g) \cong X_{2g}$ 
is
$$
 e(\Rep{\Aff{1}{\CC}}(\Sigma_g))=q^{2g-1}(q-1)^{2g}+q^{2g}-q^{2g-1}\, .
 $$
 \end{theorem}
 
 \begin{proof} 
Let 
 $$
  L=\left\{(\alpha_1, \ldots, \alpha_{s}, \beta_1, \ldots, \beta_{s}) \in \FF_{q^n}^{2s}\, \Big|\, \sum \alpha_i\beta_i=0\right\}.
  $$ 
There is a map
 $$
 \varpi:L\to \FF_{q^n}^{s}, \quad \varpi(\alpha_1, \ldots, \alpha_{s}, \beta_1, \ldots, \beta_{s})=(\alpha_1, \ldots, \alpha_{s}).
 $$
This is surjective, and $\varpi^{-1}(\alpha)$ is a hyperplane of $(\FF_{q^n})^{s}$ for $\alpha \neq (0,\ldots,0)$,
and all the space for $\alpha_0 = (0,\ldots,0)$. Hence
 \begin{align*}
 \# L &= (\# \varpi^{-1}(\alpha)) \cdot (\#(\FF_{q^n})^{s}-1) + \#(\FF_{q^n})^{s} \\
 &= (q^n)^{s-1}((q^n)^s-1)+(q^n)^s \\
 &=(q^n)^{2s-1}+ (q^n)^s-(q^n)^{s-1}\, .
 \end{align*}
 
Now, define the hyperplanes for $i = 1, \ldots, s$
 $$
 \hat{H}_i=\{(\alpha_1, \ldots, \alpha_{s}) \in \FF_{q^n}^{s} \, |\, \alpha_i=-1\}, \quad H_i = \hat{H}_i \times \FF_{q^n}^s.
 $$  
We have to remove the contributions to $L$ of these hyperplanes. Observe that $H_{i_1}\cap \ldots \cap H_{i_t} \cap L =\varpi^{-1}(\hat{H}_{i_1}\cap \ldots \cap \hat{H}_{i_t})$, for $t\geq 1$, and
in this case all fibers of $\varpi$ are hyperplanes. Thus
 $$
 \# (H_{i_1}\cap \ldots \cap H_{i_t} \cap L)= (q^n)^{2s-t-1}\, .
$$
Hence, by the inclusion-exclusion argument,
 \begin{align*}
  \# \big( (\FF_{q^n})^{2s}-  (H_{1}\cup \ldots \cup H_{s})\big) \cap L &=
   \sum_{t=0}^s (-1)^t \binom{s}{t} (q^n)^{2s-t-1} +(q^n)^s-(q^n)^{s-1}\\ & =
    (q^n)^{s-1} (q^n-1)^s +(q^n)^s-(q^n)^{s-1}\, .
 \end{align*}

This means that $\cX_s$ is strongly polynomial count with polynomial
 $$
 P_{\cX_s}(T)=T^{s-1}(T-1)^s+T^s-T^{s-1}\, .
 $$
\end{proof}

\subsection{Exhaustive polynomial count} \label{subsec:exhaust}

 There is a more computational method for finding the $E$-polynomial. 
 Suppose that we know that the variety $X$ is polynomial count. This may happen if we
 know that $X$ is Hodge-Tate type (in the sense of Remark \ref{rem:22}) 
 or that its virtual class $[X] \in \K{\Var{\CC}}$
 lies in the subring generated by the Lefschetz motive. Let $N$ be
 a bound for the dimension of $X$;
 in the case of the representation variety $\Rep{\G}(G)$, we can take $N= s\dim G -1$, where
 $s$ is the number of generators of the group $\G$.
 Then $P_X(T)$ is a polynomial of $\deg P_X \leq N$. We can count the
number of solutions to the defining equations of the variety over $\ZZ_{q_i}$, 
for a collection of $N+1$ prime powers $q_1,\ldots, q_{N+1}$. This
will determine uniquely polynomial $P_X(T)$.
  
Let us see how we can implement this idea for computing
$e(\Rep{\Aff{1}{\CC}}(\Sigma_g))$ for arbitrary genus $g$. For this, we use the quantum method explained in Section \ref{sec:5} to gain some qualitative
information on the structure of the $E$-polynomial, and the arithmetic method to actually compute
the $E$-polynomial. This is a nice combination of two methods.

As shown in Section \ref{sec:5}, the quantum method tells us that all the information
of the $E$-polynomial is encoded in a finitely generated $\ZZ[q]$-module $W$ given in (\ref{eqn:W}) and
a endomorphism $\cZ(L)$ on $W$ given in (\ref{eqn:Z}). In our case, $\dim W=2$, so in a certain basis we can write
 $$
  \cZ(L)= \begin{pmatrix} A(q) & B(q) \\ C(q) & D(q) \end{pmatrix} ,
 $$
for some polynomials $A,B,C,D \in \ZZ[q]$. The formula in Remark \ref{rem:zerodivisor} and equation (\ref{eqn:M}) tells us that we can recover the $E$-polynomial as
  \begin{equation}\label{eqn:gg}
 e(\Rep{\Aff{1}{\CC}}(\Sigma_g))= \frac{1}{q^g(q-1)^g} \begin{pmatrix} 1 & 0 \end{pmatrix}
  \begin{pmatrix} A(q) & B(q) \\ C(q) & D(q) \end{pmatrix}^g  \begin{pmatrix} 1 \\ 0 \end{pmatrix}.
  \end{equation}

Observe that the upper-left entry of $\cZ(L)^g$, which computes $e(\Rep{\Aff{1}{\CC}}(\Sigma_g))$, only depends on the product $BC$ for all $g \geq 1$. Hence, without lost of generality, we can take $C(q) = 1$. Now, observe that the first powers of $\cZ(L)$ are given by
$$
\cZ(L)^2 = \left(\begin{array}{rr}
A^{2} + B & \,\, A B + B D \\
A + D & D^{2} + B
\end{array}\right), \quad
	\cZ(L)^3 = \begin{pmatrix}
A^{3} + 2 \, A B + B D & \, \, \star \\
\star & \, \, \star
\end{pmatrix}.
$$

This implies that $A, B$ and $D$ are completely determined by the three $E$-polynomials $e(\Rep{\Aff{1}{\CC}}(\Sigma_1))$, 
$e(\Rep{\Aff{1}{\CC}}(\Sigma_2))$ and $e(\Rep{\Aff{1}{\CC}}(\Sigma_3))$, namely
$$
	A(q) = q(q-1)e(\Rep{\Aff{1}{\CC}}(\Sigma_1)), \quad B(q) = q^2(q-1)^2e(\Rep{\Aff{1}{\CC}}(\Sigma_2)) - A^2, 
$$
$$
	D(q) = \frac{q^3(q-1)^3e(\Rep{\Aff{1}{\CC}}(\Sigma_3))-A^3}{B} -2A.
$$

Now, observe that $\Rep{\Aff{1}{\CC}}(\Sigma_g)$ is an affine subvariety of $\Aff{1}{\CC}^{4g}$ so it has dimension at most $4g-1$. Hence, $e(\Rep{\Aff{1}{\CC}}(\Sigma_g))$ is a polynomial of degree at most $4g-1$ and, thus, it is completely determined by its value at $4g$ points. Since $\Rep{\Aff{1}{\CC}}(\Sigma_g)$ is polynomial counting, we can compute the number of points of $\Rep{\Aff{1}{\FF_{q_i}}}(\Sigma_g)$ for $4g$ different prime powers $q_{1}, \ldots, q_{4g}$. For that purpose, we run a small counting script \cite{GPLM-Script-2020} and we obtain the results shown in Table \ref{table:count-points}.
\begin{table}[!h]
\begin{tabular}{|c|c|c|c|c|c|c|c|c|}
\hline
	$q_i$ & $2$ & $3$ & $4$ & $5$ & $7$ & $8$ & $9$ & $11$ \\
\hline
	$g = 1$ & 4 & 18 & 48 & 100 & - & - & - & - \\\hline
	$g = 2$ & 16 & 486 & 5376 & 32500 & 446586 & 1232896 & 2991816 & 13323310  \\ \hline
	$g = 3$ & 64 & 16038 & 749568 & 12812500 & 784248234 & 3855351808 & 15479813448 &  161052610510 \\\hline
\end{tabular} 
\begin{tabular}{|c|c|c|c|c|}
\hline
	$q_i$ & $13$ & $16$ & $17$ & $19$ \\
\hline
	$g = 3$  & 1108679412828 & 11943951728640 & 23821270295824 & 84217678403958  \\\hline
\end{tabular}\caption{\label{tab:table-name}Count of points of $\Rep{\Aff{1}{\FF_{q_i}}}(\Sigma_g)$ for small prime powers $q_i$ and genus $g$.}
\label{table:count-points}
\end{table}

This implies that the corresponding $E$-polynomials are
\begin{align*}
 e(\Rep{\Aff{1}{\CC}}(\Sigma_1)) &= q^{3} - q^{2}, \\
 e(\Rep{\Aff{1}{\CC}}(\Sigma_2)) &= q^{7} - 4 q^{6} + 6 q^{5} - 3  q^{4}, \\
 e(\Rep{\Aff{1}{\CC}}(\Sigma_3)) &= q^{11} - 6  q^{10} + 15  q^{9} - 20  q^{8} + 15  q^{7} - 5  q^{6}.
\end{align*}
Therefore, we finally obtain that
$$
 \cZ(L) = \begin{pmatrix}
{\left(q - 1\right)}^{2} q^{3} & {\left(q - 1\right)}^{3} {\left(q - 2\right)}^{2} q^{6} \\
1 & {\left(q^{2} - 3 \, q + 3\right)} {\left(q - 1\right)} q^{3}
\end{pmatrix}.
$$
Plugging this matrix into equation (\ref{eqn:gg}), we recover the result of Theorem \ref{thm:44}.

\begin{remark}
The philosophy behind this method is that, with the qualitative information provided by the TQFT, the $E$-polynomial of the representation variety for arbitrary genus $g$ is completely determined by the result at small genus. And, moreover, this later value is determined by its number of points at finitely many genus and prime powers. 
\end{remark}

\section{Quantum method} \label{sec:5}

The last approach we will show for the problem of computing virtual classes of representation varieties is the so-called quantum method. 
The key idea of this method is to construct a geometric-categorical device, known as a Topological Quantum Field Theory (TQFT), and to use it for providing a precise method of computation.

\subsection{Definition of Topological Quantum Field Theories}\label{sec:brief-tqfts}

The origin of TQFTs dates back to the works of Witten \cite{Witten:1988} in which he showed that the Jones polynomial (a knot invariant) can be obtained through Chern-Simons theory, a well-known Quantum Field Theory. Aware of the importance of this discovery, Atiyah formulated in \cite{Atiyah:1988} a description of a TQFTs as a monoidal symmetric functor. This purely categorical definition is the one that we will review in this section. For a more detailed introduction, see \cite{Freed-Lurie:2010, Kock:2004}.

We will focus on \emph{symmetric monoidal categories} $(\cC, \otimes, I)$ which we recall that, by definition, are a category $\cC$ with a symmetric associative bifunctor $\otimes: \cC \times \cC \to \cC$ and a distinguished object $I \in \cC$ that acts as left and right unit for $\otimes$ (for further information, see \cite{Weibel}). A very important instance of a monoidal category is the category of $R$-modules and $R$-modules homomorphisms, $\Mod{R}$, for a given (commutative, unitary) ring $R$. The usual tensor product over $R$, $\otimes_R$, together with the ground ring $R \in \Mod{R}$ as a unit, defines a symmetric monoidal category $(\Mod{R}, \otimes_R, R)$.

In the same vein, a functor $\cF: (\cC, \otimes_{\cC}, I_\cC) \to (\cD, \otimes_{\cD}, I_\cD)$ is said to be \emph{symmetric monoidal} if it preserves the symmetric monoidal structure i.e.\ $\cF(I_\cC) = I_\cD$ and there is an isomorphism of functors
$$
	\Delta: \cF(-) \otimes_{\cD} \cF(-) \stackrel{\cong}{\Longrightarrow} \cF(- \otimes_{\cC} -).
$$

For our purposes, we will focus on the category of bordisms. Let $n \geq 1$. We define the \emph{category of $n$-bordisms}, $\CBordp{n}$, as the symmetric monoidal category given by the following data.
\begin{itemize}
	\item Objects: The objects of $\CBordp{n}$ are $(n-1)$-dimensional closed manifold, including the empty set.
	\item Morphisms: Given objects $X_1$, $X_2$ of $\CBordp{n}$, a morphism $X_1 \to X_2$ is an equivalence class of bordisms $W: X_1 \to X_2$ i.e.\ of compact $n$-dimensional manifolds with $\partial W = X_1 \sqcup X_2$. Two bordisms $W, W'$ are equivalent if there exists a diffeomorphism $F: W \to W'$ fixing the boundaries $X_1$ and $X_2$.\\
For the composition, given $W: X_1 \to X_2$ and $W': X_2 \to X_3$, we define $W' \circ W = W \cup_{X_2} W': X_1 \to X_3$ where $W \cup_{X_2} W'$ is the gluing of bordisms along $X_2$.
\end{itemize}
We endow $\CBordp{n}$ with the bifunctor given by disjoint union $\sqcup$ of both objects and bordisms. This bifunctor, with the unit $\emptyset \in \CBordp{n}$, turns $\CBordp{n}$ into a symmetric monoidal category.

\begin{definition}
Let $R$ be a commutative ring with unit. An $n$-dimensional Topological Quantum Field Theory (shortened a TQFT) is a symmetric monoidal 
functor
$$
	\cZ: \CBordp{n} \to \Mod{R}.
$$
\end{definition}

\begin{remark}
This definition slightly differs from others presented in the literature, specially in those oriented to physics, where the objects and bordisms of $\CBordp{n}$ are required to be equipped with an orientation (which plays an important role in many physical theories).
\end{remark}

The main application of TQFTs to algebraic topology comes from the following observation. Suppose that we are interested in an algebraic invariant that assigns to any closed $n$-dimensional manifold $W$ an element $\chi(W) \in R$, for a fixed ring $G$. In principle, $\chi$ might be very hard to compute and very handcrafted arguments are needed for performing explicit computations.

However, suppose that we are able to \emph{quantize} $\chi$. This means that we are able to construct a TQFT, $\cZ: \CBordp{n} \to \Mod{R}$ such that $\cZ(W)(1) = \chi(W)$ for any closed $n$-dimensional manifold. Note that the later formula makes sense since, as $W$ is a closed manifold, it can be seen as a bordism $W: \emptyset \to \emptyset$ and, since $\cZ$ is monoidal, $\cZ(W): \cZ(\emptyset) = R \to \cZ(\emptyset) = R$ is an $R$-module homomorphism and, thus, it is fully determined by the element $\cZ(W)(1) \in R$.

Such quantization gives rise to a new procedure for computing $\chi$ by decomposing $W$ into simpler pieces.
To illustrate the method, suppose that $n = 2$ and $W = \Sigma_g$ is the closed oriented surface of genus $g \geq 0$. We can decompose $\Sigma_g: \emptyset \to \emptyset$ as $\Sigma_g = D^\dag \circ L^g \circ D$, where $D: \emptyset \to S^1$ is the disc, $D^\dag: S^1 \to \emptyset$ is the opposite disc and $L: S^1 \to S^1$ is a twice holed torus, as shown in Figure \ref{img:decomp-sigmag}.

\begin{figure}[h]
	\begin{center}
	\includegraphics[scale=0.45]{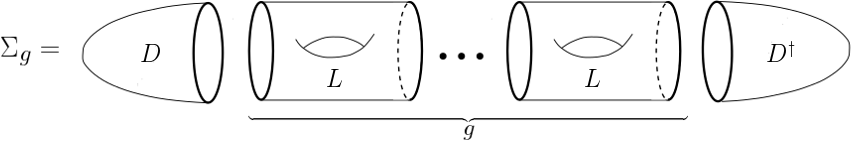}
	\caption{Decomposition of $\Sigma_g$ into simpler bordisms.}
	\label{img:decomp-sigmag}
	\end{center}
\end{figure}

In that case, applying $\cZ$ we get that
$$
	\chi(\Sigma_g) = \cZ(D^\dag) \circ \cZ(L)^g \circ \cZ(D)(1).
$$
That is, we can compute $\chi(\Sigma_g)$ for a surface of arbitrary genus just by computing three homomorphisms, $\cZ(D): R \to \cZ(S^1)$ (which is determined by an element of $\cZ(S^1)$), $\cZ(D^\dag): \cZ(S^1) \to R$ (which is essentially a projection) and an endomorphism $\cZ(L): \cZ(S^1) \to \cZ(S^1)$.

\subsection{Quantization of the virtual classes of representation varieties}\label{sec:quantization-rep-var}

The aim of this section is to quantize the virtual classes of representation varieties. However, as we will see, our construction will not give a TQFT on the nose, but a kind of lax version. 

The first ingredient we need to modify is the category of bordisms in order to include pairs of spaces. This might seem shocking at a first sight but it is very natural if we think that we are dealing with fundamental groups of topological spaces and the fundamental group is not a functor out of the category of topological spaces but out of the category of pointed topological spaces. The aim of this version for pairs is to track these basepoints.

Fix $n \geq 1$. We define the \emph{category of $n$-bordisms of pairs}, $\CBordpp{n}$ as the symmetric monoidal category given by the following data:
\begin{itemize}
	\item Objects: The objects of $\CBordpp{n}$ are pairs $(X, A)$ where $X$ is a $(n-1)$-dimensional closed manifold (maybe empty) together with a finite subset of points $A \subseteq X$ such that its intersection with each connected component of $X$ is non empty.
	\item Morphisms: Given objects $(X_1, A_1)$, $(X_2, A_2)$ of $\CBordpp{n}$, a morphism $(X_1, A_1) \to (X_2, A_2)$ is an equivalence class of pairs $(W, A)$ where $W: X_1 \to X_2$ is a bordism and $A \subseteq W$ is a finite set of points with $X_1 \cap A = A_1$ and $X_2 \cap A = A_2$. Two pairs $(W, A), (W',A')$ are equivalent if there exists a diffeomorphism of bordisms $F: W \to W'$ such that $F(A)=A'$.
Finally, given $(W, A): (X_1, A_1) \to (X_2, A_2)$ and $(W', A'): (X_2, A_2) \to (X_3, A_3)$, we define $(W',A') \circ (W,A)=(W \cup_{X_2} W', A \cup A'): (X_1, A_1) \to (X_3, A_3)$.
\end{itemize}

\begin{remark}
In this form, $\CBordpp{n}$ is not exactly a category since there is no unit morphism in $\Hom_{\CBordpp{n}} ((X,A),(X,A))$. This can be solved by weakening slightly the notion of bordism, allowing that $(X,A)$ itself could be seen as a bordism $(X,A): (X, A) \to (X,A)$.
\end{remark}

In order to construct the TQFT quantizing virtual classes of representation varieties, we need to introduce some notation. 
Fix a ground field $\bk$ (not necessarily algebraically closed) and $G$ an algebraic group over $\bk$ (not necessarily reductive).

Given a topological space $X$ and $A \subseteq X$ we denote by $\Pi(X,A)$ the fundamental groupoid of $X$ with basepoints in $A$, that is, the groupoid of homotopy classes of paths in $X$ between points in $A$.
If $X$ is compact and $A$ is finite, we define the \emph{$G$-representation variety} of the pair $(X,A)$, $\Rep{G}(X, A)$, as the set of groupoids homomorphisms $\Pi(X, A) \to G$ i.e.\ $\Rep{G}(X,A) = \Hom(\Pi(X,A), G)$. Observe that, in particular, if $A$ has a single point then $\Rep{G}(X,A)$ is the usual $G$-representation variety.

As it happened for representation varieties with a single basepoint, $\Rep{G}(X,A)$ has a natural structure of algebraic variety given as follows. Let $X = \bigsqcup\limits_{i=1}^r X_i$ be the decomposition of $X$ into connected components and let us order them so that $X_i \cap A \neq \emptyset$ for the first $s$ components. Pick $x_i \in X_i \cap A$ and, for any $i$, choose a path $\alpha_i^x$ between $x_i$ and any other $x \in X_i \cap A$, $x \neq x_i$. Then, a representation $\Pi(X, A) \to G$ is completely determined by the usual vertex representations $\pi_1(X, x_i) \to G$ for $1 \leq i \leq s$, 
together with an arbitrary element of $G$ for any chosen path $\alpha_i^x$. There are $|A|-s$ of such chosen paths, so we have a natural identification
$$
	\Rep{G}(X, A) = \prod_{i=1}^s \Rep{G}(X, x_i) \times G^{|A|-s}.
$$
The right hand side of this equality is naturally an algebraic variety, so $\Rep{G}(X, A)$ is endowed with the structure of an algebraic variety.

The second ingredient needed for quantizing representation varieties has a more algebraic nature. Given an algebraic variety $S$ over $\bk$, 
let us denote by $\Varrel{S}$ the category of algebraic varieties over $Z$, that is, the category whose objects are regular morphisms $Z \to S$ and its morphisms are regular maps $Z \to Z'$ preserving the base projections. As in the usual category of algebraic varieties, together with the disjoint union 
$\sqcup$ of algebraic varieties, and the fibered product $\times_S$ over $S$, 
we may consider its associated Grothendieck ring $\K{\Varrel{S}}$. The element of $\K{\Varrel{S}}$ induced by a morphism $h: Z \to S$ will be denoted as $[(Z, h)]_S \in \K{\Varrel{S}}$, or just by $[Z]_S$ or $[Z]$ when the morphism $h$ or the base variety are understood from the context. Recall that, in this notation, the unit of $\K{\Varrel{S}}$ is $\Unit{S} = [S, \Id_S]_S$ and that, if $S = \star$ is the singleton variety then $\K{\Varrel{\star}} = \K{\Var{\bk}}$ is the usual Grothendieck ring of varieties.

This construction exhibits some important functoriality properties that will be 
useful for our construction. Suppose that $f: S_1 \to S_2$ is a regular morphism. It induces a ring homomorphism $f^*\K{\Varrel{S_2}} \to \K{\Varrel{S_1}}$ given by $f^*[Z]_{S_2} = [Z \times_{S_2} S_1]_{S_1}$. In particular, taking the projection map $c: S \to \star$ we get a ring homomorphism $i^*: \K{\Var{\bk}} \to \K{\Varrel{S}}$ that endows the rings $\K{\Varrel{S}}$ with a natural structure of $\K{\Var{\bk}}$-module that corresponds to the cartesian product. Finally, we also have the covariant version $f_!: \K{\Varrel{S_1}} \to \K{\Varrel{S_2}}$ given by $f_![(Z, h)]_{S_1} = [(Z, f \circ h)]_{S_2}$. In general $f_!$ is not a ring homomorphism but the projection formula $f_!([Z_2] \times_{S_2} f^*[Z_1]) = f_![Z_2] \times_{S_1} [Z_1]$, for $[Z_1] \in \K{\Varrel{S_1}}$ and $[Z_2] \in \K{\Varrel{S_2}}$, implies that $f_!$ is a $\K{\Var{\bk}}$-module homomorphism. 

\begin{remark}\label{remark:properties-kvar-morph} Some important properties that clarifies the interplay between these two induced morphisms are listed below. They will be very useful for explicit computations in Section \ref{sec:application-quantum-method}. Their proof is a straightforward computation using fibered products and it can be checked in \cite{Hartshorne}. 
	\begin{itemize}
		\item The induced morphisms are functorial, in the sense that $(g \circ f)^* = f^* \circ g^*$ and $(g \circ f)_! = g_! \circ f_!$. In particular, if $i: T \hookrightarrow S$ is an inclusion, then $i^*f^* = f|_{T}^*$.
		\item Suppose that we have a pullback of algebraic varieties (i.e.\ a fibered product diagram)
	\[
\begin{displaystyle}
   \xymatrix
   {	S' = S_1 \times_S S_2 \ar[d]_{f'} \ar[r]^{\hspace{1cm} g'} & S_1 \ar[d]^f \\ S_2 \ar[r]_g & S
      }
\end{displaystyle}   
\]
Then, it holds that $g^* \circ f_! = (f')_! \circ (g')^*$. This property is usually known as the base-change formula, or the Beck-Chevalley property, and it generalizes the projection formula.
		\item Suppose that we decompose $S = T \sqcup U$, where $i: T \hookrightarrow S$ is a closed embedding and $j: U \hookrightarrow S$ is an open subvariety. Then, we have that $i_!i^* + j_! j^*: \K{\Varrel{S}} \to \K{\Varrel{S}}$ is the identity map. This corresponds to the idea that virtual classes are compatible with chopping the space according to an stratification.
	\end{itemize}
\end{remark}

At this point, we are ready to define our TQFT. We take as ground ring $R = \K{\Var{\bk}}$ the Grothendieck ring of algebraic varieties. We define a functor $\cZ: \CBordpp{n} \to \Mod{\K{\Var{\bk}}}$ as follows.
\begin{itemize}
	\item On an object $(X, A) \in \CBordpp{n}$ we set $\cZ(X,A) = \K{\Varrel{\Rep{G}(X, A)}}$, the Grothendieck ring of algebraic varieties over $\Rep{G}(X, A)$.
	\item On a morphism $(W, A): (X_1, A_1) \to (X_2, A_2)$, let us denote the natural restrictions $i: \Rep{G}(W, A) \to \Rep{G}(X_1, A_1)$ and $j: \Rep{G}(W, A) \to \Rep{G}(X_2, A_2)$. Then, we set
	$$
		\cZ(W, A) = j_! \circ i^*: \K{\Varrel{\Rep{G}(X_1, A_1)}} \to \K{\Varrel{\Rep{G}(W, A)}} \to \K{\Varrel{\Rep{G}(X_2, A_2)}}.
	$$
\end{itemize}

\begin{remark}
Recall that, since in general $j_!$ is not a ring homomorphism, the induced map $\cZ(W,A): \K{\Varrel{\Rep{G}(X_1, A_1)}} \to \K{\Varrel{\Rep{G}(X_2, A_2)}}$ is only a $\K{\Var{\bk}}$-module homomorphism.
\end{remark}

It can be proven that, since the fundamental groupoid satisfies the Seifert-van Kampen theorem, $\cZ$ is actually a functor (see \cite{GPLM-2017, GP-2018a} for a detailed proof). However, it is not monoidal since, in general, for algebraic varieties $S_1, S_2$ we have $\K{\Varrel{S_1}} \otimes_{\K{\Var{\bk}}} \K{\Varrel{S_2}} \not\cong \K{\Varrel{S_1 \times S_2}}$. Nevertheless, we still have a map
$$
	\Delta_{S_1, S_2}: \K{\Varrel{S_1}} \otimes_{\K{\Var{\bk}}} \K{\Varrel{S_2}} \to \K{\Varrel{S_1 \times S_2}}
$$
given by `external product'. That is, it is the map induced by 
$$
[Z_1] \otimes [Z_2] \in \K{\Varrel{S_1}} \otimes_{\K{\Var{\bk}}} \K{\Varrel{S_2}} \mapsto \pi_1^*[Z_1] \times_{(S_1 \times S_2)} \pi_2^*[Z_2] \in \K{\Varrel{S_1 \times S_2}},
$$
where $\pi_i: S_1 \times S_2 \to S_i$ are the projections. In this situation, it is customary to say that $\cZ$ is a \emph{symmetric lax monoidal} functor.

Finally, in order to figure our what invariant is $\cZ$ computing, first observe that for the empty set we have $\Rep{G}(\emptyset) = \star$ is the singleton variety and, thus $\cZ(\emptyset) = \K{\Varrel{\Rep{G}(\emptyset)}} = \K{\Varrel{\star}} = \K{\Var{\bk}}$ is the usual Grothendieck ring of algebraic varieties. Now, let us take $(W, A)$ a closed connected $n$-dimensional manifold. Seen as a morphism $(W, A): \emptyset \to \emptyset$, it induces a $\K{\Var{\bk}}$-module homomorphism $\cZ(W, A) = c_!c^*: \K{\Var{\bk}} \to \K{\Var{\bk}}$, where $c: \Rep{G}(W, A) \to \star$ is projection onto a point. Therefore, we have that
 \begin{align*}
	\cZ(W, A)(\Unit{\star}) & = c_!c^*(\Unit{\star}) = c_! \Unit{\Rep{G}(W, A)} = \\
	&= c_![\Rep{G}(W, A)]_{\Rep{G}(W, A)} = [\Rep{G}(W, A) ]_\star = [\Rep{G}(W, A)],
 \end{align*}
where the second equality follows from the fact that $c^*$ is a ring homomorphism. Therefore, $\cZ$ quantizes the virtual classes of representation varieties so we have proven the following result.

\begin{theorem}
Let $\bk$ be a field, $G$ an algebraic group over $k$ and $n \geq 1$. There exists a symmetric lax monoidal Topological Quantum Field Theory
$$
	\cZ: \CBordpp{n} \to \Mod{\K{\Var{\bk}}},
$$
that quantizes the virtual classes of $G$-representation varieties.
\end{theorem}

\begin{remark}\label{remark:tqft-computes-pairs}
To be precise, $\cZ$ computes virtual classes of $G$-representation varieties of pairs. This implies that it computes virtual classes of classical $G$-representation varieties up to a known constant. For instance, let $W$ be a compact connected $n$-dimensional manifold and let $A \subseteq W$ be a finite set. Then we have
$$
	\cZ(W, A)(\Unit{\star}) = [\Rep{G}(W, A)] = [\Rep{G}(W)] \times [G]^{|A|-1}.
$$
Hence, $\cZ(W, A)(\Unit{\star})$ computes $[\Rep{G}(W)]$ up to the factor $[G]^{|A|-1}$ (which is not a big problem since $[G]$ is known for most of the classical groups).
\end{remark}

Unravelling the previous construction, we can describe precisely the morphisms induced by the TQFT. Let us focus on the case $n=2$ and orientable surfaces. As we mentioned above, we need to understand the bordisms $D, D^\dag$ and $L$, as depicted in Figure \ref{img:non-parabolic-tubes}. Observe that, in order to meet the requirements of $\CBordpp{2}$, we need to chose a basepoint on $S^1$, that we will loosely denote by $\star \in S^1$. In this way $D: \emptyset \to (S^1, \star)$ and $D^\dag: (S^1, \star) \to \emptyset$ have a marked basepoint while $L: (S^1, \star) \to (S^1, \star)$ has two marked basepoints, one on each component of the boundary.

\begin{figure}[h]
	\begin{center}
	\includegraphics[scale=0.3]{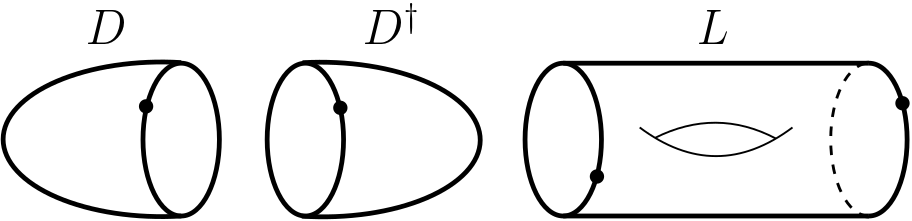}
	\caption{The basic bordisms for orientable surfaces.}
	\label{img:non-parabolic-tubes}
	\end{center}
\end{figure}

With respect to the object $(S^1, \star) \in \CBordpp{2}$, the associated representation variety is $\Rep{G}(S^1, \star) = \Hom(\ZZ, G) = G$. With respect to morphisms, the situation for $D$ and $D^\dag$ is very simple since they are simply connected. Therefore, the restriction maps at the level of fundamental groupoids are, respectively
$$
\star \longleftarrow \star \stackrel{i}{\longrightarrow} G,
	 \hspace{1cm} G \stackrel{i}{\longleftarrow} \star \longrightarrow \star,
$$
where $i: \star \hookrightarrow G$ is the inclusion of the trivial representation. Hence, under $\cZ$ we have that
$$
	\cZ(D) = i_!: \K{\Var{\bk}} \to \K{\Varrel{G}}, \hspace{1cm} \cZ(D^\dag) = i^*: \K{\Varrel{G}} \to \K{\Var{\bk}}.
$$

For the holed torus $L: (S^1, \star) \to (S^1, \star)$ the situation is a bit more complicated. Let $L = (T, A)$ where $A = \left\{x_1, x_2\right\}$ is the set of marked points of $L$, with $x_1$ in the in-going boundary and $x_2$ in the out-going boundary. Recall that $T$ is homotopically equivalent to a bouquet of three circles so its fundamental group is the free group with three generators. Thus, we can take $\gamma, \gamma_1, \gamma_2$ as the set of generators of $\pi_1(T, x_1)$ depicted in Figure \ref{img:paths-T} and $\alpha$ the path between $x_1$ and $x_2$.

\begin{figure}[h]
	\begin{center}
	\includegraphics[scale=0.15]{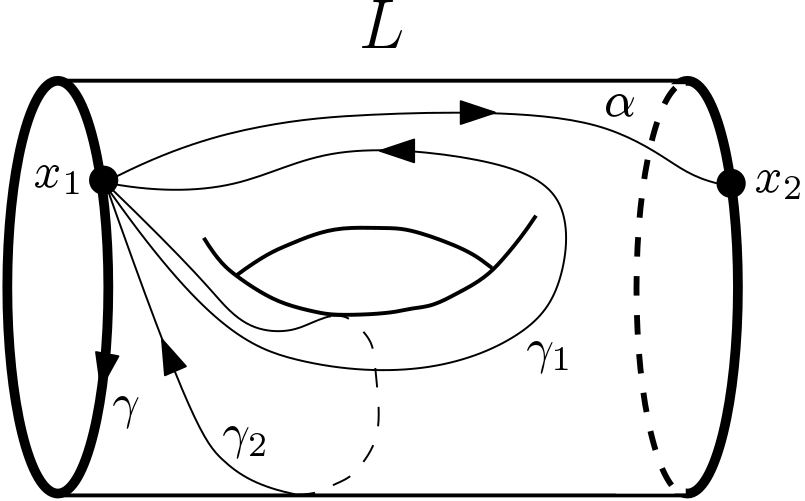}
	\caption{Chosen paths for $L$.}
	\label{img:paths-T}
	\end{center}
\end{figure}

With this description, $\gamma$ is a generator of $\pi_1(S^1, x_1)$ and $\alpha\gamma[\gamma_1, \gamma_2]\alpha^{-1}$ is a generator of $\pi_1(S^1, x_2)$, where $[\gamma_1, \gamma_2] = \gamma_1\gamma_2\gamma_1^{-1}\gamma_2^{-1}$ is the group commutator. Hence, since $\Rep{G}(L) =  \Hom(\Pi(T, A), G) = G^4$, we have that restriction maps at the level of fundamental groupoids are
$$
\begin{matrix}
	G & \stackrel{p}{\longleftarrow} & G^4 & \stackrel{q}{\longrightarrow} & G \\
	g & \mapsfrom & (g, g_1, g_2, h) & \mapsto & hg[g_1,g_2]h^{-1}
\end{matrix}
$$
where $g, g_1, g_2$ and $h$ are the images of $\gamma, \gamma_1, \gamma_2$ and $\alpha$, respectively. Hence, we obtain that
$$
	\cZ(L): \K{\Varrel{G}} \stackrel{p^*}{\longrightarrow} \K{\Varrel{G^4}} \stackrel{q_!}{\longrightarrow} \K{\Varrel{G}}.
$$

\begin{remark}\label{remark:formula-sigmag}
As we mentioned in Remark \ref{remark:tqft-computes-pairs}, the TQFT computes virtual classes of representation varieties of pairs. In particular, observe that if we decompose $\Sigma_g = D^\dag \circ L^g \circ D$, we are forced to put on $\Sigma_g$ a set of $g+1$ basepoints $A \subseteq \Sigma_g$. Hence, we have that
$$
	[\Rep{G}(\Sigma_g)] \times [G]^{g} = \cZ(\Sigma_g, A)(\Unit{\star}) = \cZ(D^\dag) \circ \cZ(L)^g \circ \cZ(D)(\Unit{\star}).
$$
Or equivalently, if we localize $\K{\Var{\bk}}$ by $[G] \in \K{\Var{\bk}}$ we have that
$$
	[\Rep{G}(\Sigma_g)] = \frac{1}{[G]^{g}} \cZ(D^\dag) \circ \cZ(L)^g \circ \cZ(D)(\Unit{\star}).
$$
\end{remark}

\subsection{Representation varieties via the quantum method} \label{sec:application-quantum-method}

In this section, as an application we will consider $G = \Aff{1}{\bk}$ and we will focus on $\Aff{1}{\bk}$-representation varieties. As in Sections 
\ref{sec:3} and \ref{sec:4}, we will compute the virtual classes of these representation varieties over any compact oriented surface but, in this case, we will use the TQFT described above for performing the computation.

As mentioned in Remark \ref{remark:formula-sigmag}, we only need to focus on the computation of the induced morphisms $\cZ(D), \cZ(D^\dag)$ and $\cZ(L)$. For the disc $\cZ(D) = i_!: \K{\Var{\bk}} \to \K{\Varrel{\Aff{1}{\bk}}}$ the situation is very simple since it is fully determined by the element $\cZ(D)(\Unit{\star}) = i_! \Unit{\star}$. Along this section, we will denote the unit of $\K{\Varrel{S}}$ by $\Unit{S}$, or just $\Unit{}$ is understood from the context. In particular $\Unit{\star} \in \K{\Var{\bk}} = \K{\Varrel{\star}}$ is the unit of the ground ring. 

In order to compute the morphism $\cZ(L): \K{\Varrel{\Aff{1}{\bk}}} \to \K{\Varrel{\Aff{1}{\bk}}}$, recall that, with the notation of Section \ref{sec:quantization-rep-var}, $\cZ(L) = q_!p^*$. We have a commutative diagram
	\[
\begin{displaystyle}
   \xymatrix
   {	& \Aff{1}{\bk}^3 \ar[r]^c \ar[d] \ar[ld]_{\varpi} & \star \ar[d]^{i} \\
   		\Aff{1}{\bk} & \Aff{1}{\bk}^4 \ar[r]_{p} \ar[l]^q & \Aff{1}{\bk}
      }
\end{displaystyle}   
\]
where $c$ is the projection onto a point, the leftmost vertical arrow is given by $(A_1, A_2, B) \mapsto (\I, A_1, A_2, B)$ and $\varpi(A_1, A_2, B) = B[A_1, A_2]B^{-1}$, being $\I \in \Aff{1}{\bk}$ the identity matrix. Moreover, the square is a pullback, so by Remark \ref{remark:properties-kvar-morph} we have
$$
	\cZ(L) \circ \cZ(D)(\Unit{\star}) = q_!p^*i_!\Unit{\star} = \varpi_! c^*\Unit{\star} = \varpi_! \Unit{\Aff{1}{\bk}^3}.
$$
In order to compute this later map, observe that, explicitly, the morphism $\varpi$ is given by
$$
	\varpi\left(\begin{pmatrix}
	a_1 & b_1\\
	0 & 1\\
\end{pmatrix}, \begin{pmatrix}
	a_2 & b_2\\
	0 & 1\\
\end{pmatrix}, \begin{pmatrix}
	x & y\\
	0 & 1\\
\end{pmatrix}\right) = \begin{pmatrix}
	1 & (a_1-1)b_2x - (a_2-1)b_1x\\
	0 & 1\\
\end{pmatrix}.
$$
Therefore, $\varpi$ is a projection onto $\ASO{1}{\bk} \subseteq \Aff{1}{\bk}$, the subgroup of orthogonal orientation-preserving affine transformations. Outside $\I \in \ASO{1}{\bk}$, $\varpi$ is a locally trivial fibration in the Zariski topology with fiber, for $\alpha \neq 0$, given by
\begin{align*}
F &= \left\{(a_1, a_2, x, b_1, b_2, y) \in (\bk^*)^3 \times \bk^3\,|\,(a_1-1)b_2x - (a_2-1)b_1x = \alpha\right\} \\
	& = \left\{b_2 = \frac{\alpha + (a_2-1)b_1x}{(a_1-1)x}, a_1 \neq 1\right\} \sqcup \left\{b_1 = -\frac{\alpha}{(a_2-1)x}, a_1 = 1\right\} \\
	& \cong \left(\left(\bk-\left\{0,1\right\}\right) \times (\bk^*)^2 \times \bk^2 \right) \sqcup \left(\bk-\set{0,1} \times \bk^* \times \bk^2 \right).
\end{align*}
Its virtual class is $[F] = (q-2)(q-1)^2q^2 + (q-2)(q-1)q^2 = q(q-1)(q^3-2q^2)$, where as always $q = [\bk] \in \K{\Var{\bk}}$.

On the other hand, on the identity matrix $\I$, the special fiber is
\begin{align*}
	\varpi^{-1}(\I) &= \left\{(a_1, a_2, x, b_1, b_2, y) \in (\bk^*)^3 \times \bk^3\,|\,(a_1-1)b_2 = (a_2-1)b_1\right\} \\
	& = \left\{b_2 = \frac{(a_2-1)b_1}{a_1-1}, a_1 \neq 1\right\} \sqcup \left\{a_1 = 1, a_2=1\right\} \sqcup \left\{a_1 = 1, a_2 \neq 1, b_1 = 0\right\} \\
	& \cong \left(\left(\bk-\left\{0,1\right\}\right) \times (\bk^*)^2 \times \bk^2 \right) \sqcup \left(\bk^* \times \bk^3\right) \sqcup \left(\bk-\set{0,1} \times \bk^* \times \bk^2 \right).
\end{align*}
Its virtual class is $[\varpi^{-1}(\I)] = (q-2)(q-1)^2q^2 + (q-1)q^3 + (q-2)(q-1)q^2 = q(q-1)(q^3-q^2)$.

Let us denote $\ASO{1}{\bk}^* = \ASO{1}{\bk} - \set{I}$ with inclusion $j: \ASO{1}{\bk}^* \hookrightarrow \Aff{1}{\bk}$. Then, by Remark \ref{remark:properties-kvar-morph}, we have that
\begin{align*}
	\varpi_!\Unit{} &= i_!i^*\varpi_! \Unit{} + j_!j^*\varpi_! \Unit{} =  i_!(\varpi|_{\varpi^{-1}(\I)})_! \Unit{} + j_!(\varpi|_{\varpi^{-1}(\ASO{1}{\bk}^*)})_! \Unit{}.
\end{align*}

For the first map, recall that $\varpi$ is locally trivial in the Zariski topology over $\ASO{1}{\bk}^*$. Thus, $(\varpi|_{\varpi^{-1}(\ASO{1}{\bk}^*)})_! \Unit{\Aff{1}{\bk}^3} = [F]\, \Unit{\ASO{1}{\bk}^*}$. On the other hand, the map $\varpi|_{\varpi^{-1}(\I)}$ is projection onto a point so $(\varpi|_{\varpi^{-1}(\I)})_! \Unit{\Aff{1}{\bk}^3} = [\varpi^{-1}(\I)]\Unit{\star}$. Hence, putting all together, we obtain that
\begin{align*}
	\cZ(L) \circ \cZ(D)(\Unit{\star}) &=  i_!(\varpi|_{\varpi^{-1}(\I)})_! \Unit{} + j_!(\varpi|_{\varpi^{-1}(\ASO{1}{\bk}^*)})_! \Unit{} \\
	&= q(q-1)(q^3-q^2) \,i_!\Unit{\star} + q(q-1)(q^3-2q^2)\,j_!\Unit{\ASO{1}{\bk}^*}.
\end{align*}

In this way, if we want to apply $\cZ(L)$ twice, we need to compute the image $\cZ(L)(j_!\Unit{\ASO{1}{\bk}^*})$. This computation is quite similar to the previous one. First, we again have a commutative diagram whose square is a pullback
	\[
\begin{displaystyle}
   \xymatrix
   {	& \ASO{1}{\bk}^* \times \Aff{1}{\bk}^3 \ar[r] \ar[d] \ar[ld]_{\vartheta} & \ASO{1}{\bk}^* \ar[d]^{j} \\
   		\Aff{1}{\bk} & \Aff{1}{\bk}^4 \ar[r]_{p} \ar[l]^q & \Aff{1}{\bk}
      }
\end{displaystyle}  
\]
The leftmost vertical arrow is the inclusion map and $\vartheta(A, A_1, A_2, B) = BA[A_1, A_2]B^{-1}$. Computing explicitly, we have that
$$
	\vartheta\left(\begin{pmatrix}
	1 & \beta\\
	0 & 1\\
\end{pmatrix}, \begin{pmatrix}
	a_1 & b_1\\
	0 & 1\\
\end{pmatrix}, \begin{pmatrix}
	a_2 & b_2\\
	0 & 1\\
\end{pmatrix}, \begin{pmatrix}
	x & y\\
	0 & 1\\
\end{pmatrix}\right) = \begin{pmatrix}
	1 \, \, & (a_1-1)b_2x - (a_2-1)b_1x + \beta x\\
	0 \, \, & 1\\
\end{pmatrix}.
$$
Hence, $\vartheta$ is again a morphism onto $\ASO{1}{\bk} \subseteq \Aff{1}{\bk}$. Over $\I \in \ASO{1}{\bk}$, the fiber is
\begin{align*}
	\vartheta^{-1}(\I) &= \left\{(\beta, a_1, a_2, x, b_1, b_2, y) \in (\bk^*)^4 \times \bk^3\,|\, (a_2-1)b_1x - (a_1-1)b_2x = \beta\right\} \\
	&= \left((\bk^*)^3 \times \bk^3\right)- \left\{(a_1-1)b_2 - (a_2-1)b_1 = 0\right\} \\ & = \left((\bk^*)^3 \times \bk^3\right)-\varpi^{-1}(\I).
\end{align*}
Thus, $[\vartheta^{-1}(\I)] = (q-1)^3q^3-q(q-1)(q^3-q^2) = q(q-1)(q^4-3q^3+2q^2)$.

Analogously, on $\ASO{1}{\bk}^*$, we have that $\vartheta$ is a locally trivial fibration in the Zariski topology with fiber over $\alpha \neq 0$ given by
\begin{align*}
	F' &= \left\{(\beta, a_1, a_2, x, b_1, b_2, y) \in (\bk^*)^4 \times \bk^3\,|\,(a_1-1)b_2x - (a_2-1)b_1x + \beta = \alpha\right\} \\
	&= \left((\bk^*)^3 \times \bk^3\right)- \left\{(a_1-1)b_2 - (a_2-1)b_1 = \alpha\right\} = \left((\bk^*)^3 \times \bk^3\right)-F.
\end{align*}
Hence, the virtual class of the fiber is $[F'] = (q-1)^3q^3-q(q-1)(q^3-2q^2) = q(q-1)(q^4-3q^3+3q^2)$. Putting together these computations we obtain that
\begin{align*}
	\cZ(L) & \left(  j_!\Unit{\ASO{1}{\bk}^*}\right) = \vartheta_!\Unit{} =  i_!(\vartheta|_{\vartheta^{-1}(\I)})_!\Unit{} + j_!(\vartheta|_{\vartheta^{-1}(\ASO{1}{\bk}^*)})_! \Unit{} \\
	&= q(q-1)(q^4-3q^3+2q^2) \,i_!\Unit{\star} + q(q-1)(q^4-3q^3+3q^2)\,j_!\Unit{\ASO{1}{\bk}^*}.
\end{align*}

Let $W \subseteq \K{\Varrel{\Aff{1}{\bk}}}$ be the submodule generated by the elements $i_!\Unit{\star}$ and $j_!\Unit{\ASO{1}{\bk}^*}$. The previous computation shows that $\cZ(L)(W) \subseteq W$. Furthermore, indeed we have 
 \begin{equation}\label{eqn:W}
 W = \langle \cZ(L)^g(i_!\Unit{\star})\rangle_{g=0}^\infty\, .
 \end{equation} 
 On $W$, the morphism $\cZ(D^\dag): W \to \K{\MHS{\QQ}}$ is given by the projection $\cZ(D^\dag)(i_!\Unit{\star}) = \Unit{\star}$ and $\cZ(D^\dag)(j_!\Unit{\ASO{1}{\bk}^*}) = 0$. Hence, regarding the computation of virtual classes of representation varieties, we can restrict our attention to $W$.

If we want to compute explicitly these classes, observe that, by the previous calculations, on the set of generators $i_!\Unit{\star}, j_!\Unit{\ASO{1}{\bk}^*}$ of $W$, the matrix of $\cZ(L): W \to W$ is
 \begin{equation}\label{eqn:Z}
\cZ(L) = q(q-1) \begin{pmatrix}
	q^3-q^2 & q^4-3q^3+2q^2 \\
	q^3-2q^2 & q^4-3q^3+3q^2\\
\end{pmatrix}.
 \end{equation} 

Since $[\Aff{1}{\bk}] = [\bk^* \times \bk] = q(q-1)$, using the formula of Remark \ref{remark:formula-sigmag}, we obtain that
\begin{equation}\label{eqn:M}
\begin{aligned} \,
	[\Rep{\Aff{1}{\bk}} & (\Sigma_g)] = \begin{pmatrix}
	1 & 0\\
\end{pmatrix}
\begin{pmatrix}
	q^3-q^2 & q^4-3q^3+2q^2 \\
	q^3-2q^2 & q^4-3q^3+3q^2\\
\end{pmatrix}^g\begin{pmatrix}
	1\\
	0\\
\end{pmatrix} \\
&= \begin{pmatrix}
	1 & 0\\
\end{pmatrix}
\begin{pmatrix}
	q-1 & q-1 \\
	-1 & q-1\\
\end{pmatrix}
\begin{pmatrix}
	q^{2g} & 0 \\
	0 & q^{2g}(q-1)^{2g}\\
\end{pmatrix}\begin{pmatrix}
	q-1 & q-1 \\
	-1 & q-1\\
\end{pmatrix}^{-1}\begin{pmatrix}
	1\\
	0\\
\end{pmatrix} \\
&=q^{2g-1}\left((q-1)^{2g}+q-1\right).
\end{aligned}
\end{equation}

\begin{remark} \label{rem:14}
Strictly speaking, this is not the virtual class of $\Rep{\Aff{1}{\bk}}(\Sigma_g)$ on $\K{\Var{\bk}}$ but on its localization by the multiplicative set $S$ generated by $q$ and $q-1$. This has some peculiarities since, as mentioned in Remark \ref{rem:zerodivisor}, 
$q = [\CC]$ is a zero divisor of $\K{\Var{\bk}}$. Hence, the morphism $\K{\Var{\bk}} \to S^{-1}\K{\Var{\bk}}$ is not injective and indeed, its kernel is the annihilator of $q$ or $q-1$. In this way, strictly we have computed the virtual class of the representation variety up to annihilators of $q$ or $q-1$. This is a common feature of the quantum method, due to the requirement of Remark \ref{remark:formula-sigmag} of inverting $[G]$.
\end{remark}

\subsection{Concluding remarks}

The previous calculation agrees with the one of Sections \ref{sec:3} and \ref{sec:4}. It may seem that this quantum approach is lengthier than the other methods, but its strength lies in on the fact that it does not depend on finding good geometric descriptions. Therefore, it offers a systematic method that can be applied to more general contexts in which geometric or arithmetic methods fail. For instance, in \cite{GP-2019}, it is computed the virtual classes of $\SL_2(\CC)$-parabolic representation varieties in the general case by means of the quantum method. This result is unavailable using the geometric or the arithmetic approach due to very subtle interaction between the monodromies of the punctures that cannot be captured with the classical methods.

This calculation also shows a general feature of the quantum method. In principle, the $\K{\Var{\bk}}$-module $\cZ(S^1, \star) = \K{\Varrel{G}}$, in which we have to perform the computations, is infinitely generated. However, in all the known computations of $\cZ$, it turn out that the computation can be restricted to a certain finitely generated submodule $W \subseteq \cZ(S^1, \star)$ as it happened above.

This fact that $\cZ(S^1, \star)$ is infinitely generated is in sharp contract with what happens for strict monoidal TQFTs. For $\cZ$ a monoidal TQFT, a straightforward duality argument shows that $\cZ(X)$ is forced to be a finitely generated module (see \cite{Kock:2004}). Indeed, this observation is the starting point of the later developments towards the classification of extended TQFTs \cite{Lurie:2009}, that show that the whole TQFT is determined by this `fully dualizable' object.

In this sense, the lax monoidal TQFT for representation varieties exhibits a mixed behaviour, since it takes values in an infinitely generated module but the calculations can be performed in a finitely submodule, mimicking an strict monoidal TQFT. On the other hand, when dealing with parabolic character varieties, the TQFT quantizing representation varieties is intrinsically infinitely generated. Definitely, further research is needed for shedding light to the interplay between lax monoidal and strict monoidal TQFTs.

\end{document}